\documentclass[hidelinks]{amsart}

\usepackage{xcolor,enumerate}
\definecolor{darkgreen}{rgb}{0,0.5,0}

\usepackage[T1]{fontenc}
\usepackage{hyperref}
\usepackage{url}
\usepackage{booktabs}
\usepackage{amsfonts}
\usepackage{nicefrac}
\usepackage{microtype}
\usepackage{amssymb,latexsym,amsfonts,amsmath}
\usepackage{graphicx}
\usepackage{textcomp}
\usepackage{subcaption}
\usepackage{color}
\usepackage{mathabx}
\usepackage{algorithm,algorithmic}
\usepackage{bigints}
\usepackage{mathtools}
\usepackage{tikz}
\usetikzlibrary{arrows.meta,positioning,fit,calc}

\DeclareMathOperator{\tr}{tr}

\newcommand{\real}{{\mathbb{R}}}
\newcommand{\R}{{\mathbb{R}}}
\newcommand{\E}{{\mathbb{E}}}
\newcommand{\F}{{\mathcal{F}}}

\newcommand{\PP}{{\mathbb{P}}}

\newtheorem{theorem}{Theorem}[section]
\newtheorem{corollary}{Corollary}[section]
\newtheorem{proposition}{Proposition}[section]
\newtheorem{lemma}{Lemma}[section]
\newtheorem{sublemma}{Sublemma}[section]
\theoremstyle{definition}

\newtheorem{remark}{Remark}[section]

\newcommand{\eps}{\epsilon}
\newcommand{\argmin}{\operatorname{argmin}}

\newcommand{\longthmtitle}[1]{\mbox{}\textup{\textbf{(#1):}}}

\newcommand{\G}{\mathcal{G}}
\newcommand{\D}{\mathcal{D}}

\newcommand{\N}{\mathcal{N}}

\newcommand{\oprocendsymbol}{\hbox{$\diamond$}}
\newcommand{\oprocend}{\relax\ifmmode\else\unskip\hfill\fi\oprocendsymbol}
\newcommand{\proofend}{\hfill\raisebox{0.15ex}{\scalebox{0.85}{$\blacksquare$}}}

\begin{document}

\begin{abstract}
Inspired by REINFORCE, we introduce a novel receding-horizon algorithm for the Linear Quadratic Regulator (LQR) problem with unknown dynamics. Unlike prior methods, our algorithm avoids reliance on two-point gradient estimates while maintaining the same order of sample complexity. Furthermore, it eliminates the restrictive requirement of starting with a stable initial policy, broadening its applicability. Beyond these improvements, we introduce a refined analysis of error propagation through the contraction of the Riccati operator under the Riemannian distance. This refinement leads to a better sample complexity and ensures improved convergence guarantees. 
\end{abstract}

\title{Sample Complexity of Linear Quadratic Regulator Without Initial Stability}

\author[Amirreza Neshaei Moghaddam]{Amirreza Neshaei Moghaddam}
\address{Department of Electrical and Computer Engineering\\
University of California at Los Angeles,
Los Angeles}
\email{amirnesha@ucla.edu}
\author[Alex Olshevsky]{Alex Olshevsky}
\address{Department of Electrical and Computer Engineering\\
       Boston University}
\email{alexols@bu.edu}
\author[Bahman Gharesifard]{Bahman Gharesifard}
\address{Department of Electrical and Computer Engineering\\
       University of California, Los Angeles}
\email{gharesifard@ucla.edu}


\maketitle

\section{Introduction}
The Linear Quadratic Regulator (LQR) problem, a cornerstone of optimal control theory, offers an analytically tractable framework for optimal control of linear systems with quadratic costs. Traditional methods rely on complete knowledge of system dynamics, solving the Algebraic Riccati Equation~\cite{DPB:95} to determine optimal control policies. However, recent real-world scenarios often involve incomplete or inaccurate models. Classical methods in control theory, such as identification theory~\cite{MG:06} and adaptive control~\cite{AMA-ALF:21}, were specifically designed to provide guarantees for decision-making in scenarios with unknown parameters. However, the problem of effectively approximating the optimal policy using these methods remains underexplored in the traditional literature. Recent efforts have sought to bridge this gap by analyzing the sample complexity of learning-based approaches to LQR~\cite{SD-HM-NM-BR-ST:20}, providing bounds on control performance relative to the amount of data available.

In contrast, the model-free approach, rooted in reinforcement learning (RL), bypasses the need for explicit dynamics identification, instead focusing on direct policy optimization through cost evaluations. Recent advances leverage stochastic zero-order optimization techniques, including policy gradient methods, to achieve provable convergence to near-optimal solutions despite the inherent non-convexity of the LQR cost landscape. Foundational works, such as~\cite{MF-RG-SK-MM:18}, established the feasibility of such methods, demonstrating convergence to the globally optimal solution. Subsequent efforts, including~\cite{DM-AP-KB-KK-PLB-MJW:20} and~\cite{HM-AZ-MS-MRJ:22}, have refined these techniques, achieving improved sample complexity bounds. Notably, all of these works assume that the initial policy is stabilizing.

A key limitation of these methods, including~\cite{DM-AP-KB-KK-PLB-MJW:20, HM-AZ-MS-MRJ:22}, is the reliance on two-point gradient estimation, which requires evaluating costs for two different policies while maintaining identical initial states. In practice, this assumption is often infeasible, as the initial state is typically chosen randomly and cannot be controlled externally. Our earlier work~\cite{ANM-AO-BG:25} addressed this challenge, establishing the best-known result among methods that assume initial stability without having to rely on two-point estimates. Instead, we proposed a gradient estimation method inspired by REINFORCE~\cite{RJW:92, RSS-DM-SS-YM:99}, that achieves the same convergence rate as the two-point method~\cite{DM-AP-KB-KK-PLB-MJW:20} while using only a single cost evaluation at each step. This approach enhances the practical applicability of model-free methods, setting a new benchmark under the initial stability assumption.

The requirement for an initial stabilizing policy significantly limits the utility of these methods in practice. Finding such a policy can be challenging or infeasible and often relies on identification techniques, which model-free methods are designed to avoid. Without getting technical at this point, it is worth pointing out that this initial stability assumption plays a major role in the construction of the mentioned model-free methods, and cannot be removed easily. For instance, this assumption ensures favorable optimization properties, like coercivity and gradient domination, that simplify convergence analysis. In this sense, removing this assumption while maintaining stability and convergence guarantees is essential to generalize policy gradient methods, a challenge that has remained an active research topic~\cite{XZ-TB:23, JCP-JU-MS:21, AL:20, FZ-XF-KY:21}.

As also pointed out in~\cite{XZ-TB:23}, the $\gamma$-discounted LQR problems studied in~\cite{JCP-JU-MS:21, AL:20, FZ-XF-KY:21} are equivalent to the standard undiscounted LQR with system matrices scaled by $\sqrt{\gamma}$. In~\cite{JCP-JU-MS:21, AL:20, FZ-XF-KY:21}, this scaling results in an enlarged set of stabilizing policies when $\gamma$ is sufficiently small, enabling policy gradient algorithms to start from an arbitrary policy by choosing $\gamma$ appropriately. 
However, in this setting, the discounting factor serves as an additional parameter that can be tuned (e.g., as in~\cite{JCP-JU-MS:21}). 
In contrast, receding-horizon approaches~\cite{XZ-TB:23}, and our work here, consider the undiscounted LQR setting, where no discount factor is introduced, and if one were present, it would be treated as a fixed problem parameter rather than a tunable one. 

This paper builds on the receding-horizon policy gradient framework introduced in~\cite{XZ-TB:23}, a significant step towards eliminating the need for a stabilizing initial policy by recursively updating finite-horizon costs. While the approach proposed in~\cite{XZ-TB:23} marks an important step forward in model-free LQR, we address the reliance on the two-point gradient estimation, a known limitation discussed earlier. 
Building on the gradient estimation approach from our earlier work~\cite{ANM-AO-BG:25}, we adapt the core idea to accommodate the new setup that eliminates the initial stability assumption. Specifically, our modified method retains the same convergence rate while overcoming the restrictive assumptions of two-point estimation. Beyond these modifications, we revisit the convergence analysis in~\cite{XZ-TB:23} and identify a calculation error that obscured the fact that their error accumulation grows exponentially with the horizon length, affecting the claimed sample complexity guarantee. We also provide the corresponding corrected sample complexity rate. We then develop a refined argument based on the Riemannian distance~\cite{RB:07}, under which the accumulated error grows only linearly with the horizon length rather than exponentially. This leads to a sample complexity bound of $\widetilde{\mathcal{O}}(\eps^{-2})$, with no problem-specific constants appearing in the exponent of $(1/\eps)$.
 
In summary, the key contributions of this work are:
\begin{itemize}
    \item We develop a new error accumulation analysis under which the accumulated error grows linearly, rather than exponentially, with the horizon length, yielding an improved sample complexity bound of $\widetilde{\mathcal{O}}(\eps^{-2})$.
    \item We propose a more practical gradient estimate that avoids reliance on two-point estimation and its assumption of the ability to select the randomness of system.
\end{itemize}

\subsection{Algorithm and Paper Structure Overview}\label{sub: structure}
The paper is structured into three sections. Section II presents the necessary preliminaries and establishes the notation used throughout the paper. Section III introduces our proposed algorithm, which operates through a hierarchical double-loop structure, an outer loop which provides a surrogate cost function in a receding horizon manner, and an inner loop applying policy gradient method to obtain an estimate of its optimal policy. Section IV delves deeper into the policy gradient method employed in the inner loop, providing rigorous convergence results and theoretical guarantees for this critical component of the algorithm. Section~V includes the sample complexity bounds, and comparisons with the results in the literature. Finally, we provide simulations studies verifying our findings in Section~VI. 

To be more specific, the core idea of the algorithm leverages the observation that, for any error tolerance $\eps$, there exists sufficiently large finite horizon $N$ where the sequence of policies minimizing recursively updated finite-horizon costs can approximate the optimal policy for the infinite-horizon cost within $\eps$ neighborhood. This receding-horizon design is key to eliminating the need for an initial stabilizing policy: because each horizon step solves a finite-horizon problem, all stage costs remain finite even when the initial policy is unstable. This insight motivates the algorithm’s design: a backward outer loop that iteratively constructs surrogate cost functions over a sequence of finite horizons, and a forward inner loop that optimizes these surrogate costs. At each stage $h$, the policies for later steps $h+1$ to $N-1$ are frozen, and the inner loop applies a policy gradient method to optimize only over the policy at stage $h$, yielding an approximately optimal policy $\widetilde{K}_{h}$ for the finite-horizon cost for the interval $[h,N]$. The resulting policy $\widetilde{K}_{h}$ is then incorporated into the cost function of the subsequent step in the outer loop, and the process repeats backward until $h=0$. Moreover, our inner loop uses a model-free, one-point gradient estimator, which relies solely on observed costs and does not require knowledge of the system dynamics. Finally, note that the terminology ``receding-horizon'' here refers to this backward policy-computation procedure, rather than an online receding-horizon implementation; in particular, the final output $\widetilde{K}_0$ is a single time-invariant feedback gain that will be applied to the original infinite-horizon LQR problem.

The main difficulty in analyzing the proposed algorithm stems from the fact that the approximation errors from the policy gradient method in the inner loop propagate across all steps of the outer loop. To ensure overall convergence, the algorithm imposes a requirement on the accuracy of the policy gradient method in the inner loop. Each policy obtained must be sufficiently close to the optimal policy for the corresponding finite-horizon cost. This guarantees that the final policy at the last step of the outer loop converges to the true optimal policy for the infinite-horizon cost.

\section{Preliminaries}
\label{sec:preliminaries}
In this section, we gather the required notation, closely following the ones in~\cite{XZ-TB:23} which our work builds on. We use $\| X \|$, $\| X \|_F$, $\sigma_{\min} (X)$, and $\kappa_{X}$ to denote the 2-norm, Frobenius norm, minimum singular value, and the condition number of a matrix $X$ respectively. We also use $\rho (X)$ to denote the spectral radius of a square matrix $X$. Moreover, for a positive definite matrix $W$ of appropriate dimensions, we define the $W$-induced norm of a matrix $X$ as
\[
\| X \|_W^2 := \sup_{z \neq 0} \frac{z^\top X^\top W X z}{z^\top W z}.
\]
Furthermore, we denote the Riemannian distance~\cite{RB:07} between two symmetric positive definite matrices $U,V \in \R^{n \times n}$ by
\[
\delta(U,V) := \left( \sum_{k=1}^{n} \log^2 \lambda_k(U V^{-1}) \right)^{1/2}.
\]
Now consider the discrete-time linear system
\begin{equation}
x_{t+1} = A x_t + B u_t, \label{eq: LTI_discrete_time}
\end{equation}
where $x_t \in \real^n$ is the system state at time $t \in \{0,1,2,\dotsc\}$, $u_t \in \real^{m}$ is the control input at time $t$, $A \in \real^{n \times n}$ and $B \in \real^{n \times m}$ are the system matrices which are \emph{unknown} to the control designer. In our setting, the initial state $x_0$ is drawn randomly from a distribution $\D$ and satisfies 
\begin{align}
    \E[x_0] = 0, \quad 
    \E[x_0 x_0^\top] = \Sigma_0, \ \text{and} \ \| x_0 \|^2 \leq C_m \ \text{a.s.} \label{eq: noise_assumptions}
\end{align}
for some positive definite $\Sigma_0 \in \real^{n \times n}$ and scalar $C_m > 0$. This assumption of a random initial state has been widely adopted in the recent LQR literature~\cite{MF-RG-SK-MM:18,DM-AP-KB-KK-PLB-MJW:20,HM-AZ-MS-MRJ:22, CJ-GK-GL:25,BH-RX-HY:21, ML-JQ-WXZ-YW-YK:22,PC-MS-PP:20,JL-JX-ZS:23}.
The objective in the LQR problem is to find the optimal controller that minimizes the following cost
\[
J_{\infty} = \E \left[\sum_{t=0}^{\infty} x_t^\top Q x_t + u_t^\top R u_t \right],
\]
where $Q \in \real^{n \times n}$ and $R \in \real^{m \times m}$ are the symmetric positive definite matrices that parameterize the cost. We require the pair $(A,B)$ to be stabilizable, and since $Q$ is positive definite (which we also show by $Q \succ 0$), the pair $(A,Q^{1/2})$ is observable.
As a result, the unique optimal controller is a linear state-feedback $u^\ast_t = - K^* x_t$ where $K^* \in \real^{m \times n}$ is derived as follows
\begin{equation}
K^* = (R + B^\top P^* B)^{-1} B^\top P^* A, \label{eq: K* explicit formulation}
\end{equation}
and $P^* \in \real^{n \times n}$ denotes the unique positive definite solution to the undiscounted discrete-time algebraic Riccati equation (ARE)~\cite{DPB:95}:
\begin{equation}
    P = A^\top P A - A^\top P B (R + B^\top P B)^{-1} B^\top P A + Q. \label{eq: ARE}
\end{equation}
Moreover, following the notation in~\cite{XZ-TB:23}, we denote the $P^*$-induced norm by $\| X \|_*$. Since the optimal controller is a linear state feedback, we formulate the costs directly in terms of the feedback gain $K$, taking $u_t = -Kx_t$ for all $t \geq 0$. 
Throughout the paper, we use the terms ``policy'' and ``feedback gain'' interchangeably to refer to such matrices $K$.

We now introduce important notation used in the algorithm description and the proof of the main result.
Let \( N \) be the horizon length and \( h \) the initial time step. The true finite-horizon cost \( J_h(K_h) \) of a policy \( K_h \) is defined as
\begin{align}
    J_h(K_h) := &\E_{x_0 \sim \mathcal{D}} \Bigg[ \sum_{t=h+1}^{N-1} x_{t-h}^\top \left( Q + (K^*_t)^\top R K^*_t \right) x_{t-h} + x_0^\top \left( Q + K_h^\top R K_h \right) x_0 + x_{N-h}^\top Q_N x_{N-h} \Bigg] , \label{eq: J_h_definition}
\end{align}
where:
\begin{itemize}
    \item \( x_0 \) denotes the initial state \( x_0 \) drawn from distribution $\mathcal{D}$,
    \item \( Q_N \) is the terminal cost matrix,
    \item \( K_h \) is the feedback gain applied at step \( 0 \),
    \item \( K^*_t \) is the feedback gain at step \( t-h \), to be formally defined via the Riccati difference equation in~\eqref{eq: RDE_K};   
\end{itemize}
Finally, for all \( t \in \{h+1, \dotsc, N-1\} \), the state evolves as
    \[
    x_{t-h+1} = (A - BK^*_t) x_{t-h},
    \]
    with
    \[
    x_{1} = (A - BK_h) x_0.
    \]
While $J_h$ is the true finite-horizon cost, its downstream policies $\{ K^*_{h+1},\dotsc,K^*_{N-1} \}$ are obtained from the RDE and cannot be computed directly when the system matrices are unknown. Therefore, we replace them with the policies generated by the algorithm and define the surrogate cost
\begin{align}
    \widetilde{J}_h(K_h) := &\E_{x_0 \sim \mathcal{D}} \Bigg[ \sum_{t=h+1}^{N-1} x_{t-h}^\top \left( Q + \widetilde{K}_t^\top R \widetilde{K}_t \right) x_{t-h} + x_0^\top \left( Q + K_h^\top R K_h \right) x_0 + x_{N-h}^\top Q_N x_{N-h} \Bigg], \label{eq: J_h_tilde_definition}
\end{align}
where \( \widetilde{K}_t \) is the feedback gain derived at step \( t \) of the outer loop of the algorithm, and for all \( t \in \{h+1, \dotsc, N-1\} \), the state evolves as:
    \[
    x_{t-h+1} = (A - B\widetilde{K}_t) x_{t-h},
    \]
    with
    \[
    x_{1} = (A - BK_h) x_0.\]    
The key difference between \( \widetilde{J}_h(K_h) \) and \( J_h(K_h) \) lies in the use of \( \widetilde{K}_t \) versus \( K^*_t \) for \( t \in \{ h+1, \dotsc, N-1 \} \). This distinction implies that \( \widetilde{J}_h(K_h) \) incorporates all errors from earlier steps, precisely the ones at \( \{N-1, \dotsc, h+1\} \), as the procedure is backward.

We now define several functions that facilitate the characterization of our gradient estimate, which uses ideas from our earlier work in~\cite{ANM-AO-BG:25}). To start, we let
\begin{align}
    \widetilde{J}_h (K_h; x_0) &:= x_0^\top (Q + K_h^\top R K_h ) x_0 + \sum_{t=h+1}^{N-1} x_{t-h}^\top \left( Q + \widetilde{K}_t^\top R \widetilde{K}_t \right) x_{t-h} + x_{N-h}^\top Q_N x_{N-h} \cr
    &= x_0^\top (Q + K_h^\top R K_h ) x_0 + x_0^\top (A-BK_h)^\top \widetilde{P}_{h+1} (A-BK_h) x_0, \qquad \label{eq: J_tilde_init2}
\end{align}
so that
\[
\widetilde{J}_h (K_h) = \E_{x_0 \sim \mathcal{D}} \left[\widetilde{J}_h (K_h; x_0)\right].
\]
Using~\eqref{eq: J_tilde_init2}, we can compute the gradient of $\widetilde{J}_h (K_h; x_0)$ with respect to $K_h$ as follows:
\begin{align}
    \nabla \widetilde{J}_h (K_h; x_0) &= \nabla \left( x_0^\top K_h^\top R K_h x_0 + x_0^\top K_h^\top B^\top \widetilde{P}_{h+1} B K_h x_0 - 2 x_0^\top A^\top \widetilde{P}_{h+1} B K_h x_0 \right) \cr
    &= 2 R K_h x_0 x_0^\top + 2 B^\top \widetilde{P}_{h+1} B K_h x_0 x_0^\top - 2 B^\top \widetilde{P}_{h+1} A x_0 x_0^\top \cr
    &= 2 \left( (R + B^\top \widetilde{P}_{h+1} B) K_h - B^\top \widetilde{P}_{h+1} A \right) x_0 x_0^\top,
\end{align}
and thus,
\begin{align}
    \nabla \widetilde{J}_h (K_h) &= \E_{x_0 \sim \mathcal{D}} \left[ \nabla \widetilde{J}_h (K_h; x_0) \right] \cr
    &= 2 \left( (R + B^\top \widetilde{P}_{h+1} B) K_h - B^\top \widetilde{P}_{h+1} A \right) \E_{x_0 \sim \mathcal{D}} \left[ x_0 x_0^\top \right] \cr
    &= 2 \left( (R + B^\top \widetilde{P}_{h+1} B) K_h - B^\top \widetilde{P}_{h+1} A \right) \Sigma_0. \label{eq: actual_gradient}
\end{align}
Moreover, we define
\begin{align}
    Q_h(x_0, u_0) &:= x_0^\top Q x_0 + u_0^\top R u_0 + \sum_{t=h+1}^{N-1} x_{t-h}^\top \left( Q + \widetilde{K}_t^\top R \widetilde{K}_t \right) x_{t-h} + x_{N-h}^\top Q_N x_{N-h} \cr
    &= x_0^\top Q x_0 + u_0^\top R u_0 +\widetilde{J}_{h+1} (\widetilde{K}_{h+1}; Ax_0 + Bu_0) \cr
    &= x_0^\top Q x_0 + u_0^\top R u_0 + (Ax_0 + Bu_0)^\top \widetilde{P}_{h+1} (Ax_0 + Bu_0), \label{eq: Q_func_def}
\end{align}
so that
\[
\widetilde{J}_h (K_h; x_0) = Q_h(x_0, - K_h x_0),
\]
and
\[
\widetilde{J}_h (K_h) = \E_{x_0 \sim \mathcal{D}} \left[ Q_h(x_0, - K_h x_0) \right].
\]
Having established the cost functions, we now introduce the notation used to describe the policies:
\begin{align}
    K^*_h &:= \argmin_{K_h} J_h(K_h), \cr \widetilde{K}^*_h &:= \argmin_{K_h} \widetilde{J}_h(K_h), \label{eq: K_argmin_def}
\end{align}
where \( K^*_h \) denotes the optimal policy for the true cost \( J_h(K_h) \), and \( \widetilde{K}^*_h \) denotes the optimal policy for the surrogate cost \( \widetilde{J}_h(K_h) \).
Additionally, \( \widetilde{K}_h \) represents an estimate of \( \widetilde{K}^*_h \). It is obtained using a policy gradient method in the inner loop of the algorithm, which is applied at each step \( h \) of the outer loop to minimize the surrogate cost \( \widetilde{J}_h(K_h) \).

We now move on to the recursive equations. First, we have
\begin{align}
    \widetilde{P}_{h} &= (A - B \widetilde{K}_h)^\top \widetilde{P}_{h+1} (A - B \widetilde{K}_h) + \widetilde{K}_h^\top R \widetilde{K}_h + Q, \label{eq: lqr_lyapunov}
\end{align}
where $\widetilde{P}_N = Q_N$. In addition,
\begin{align}
    \widetilde{P}^*_{h} &= (A - B \widetilde{K}^*_h)^\top \widetilde{P}_{h+1} (A - B \widetilde{K}^*_h) + (\widetilde{K}^*_h)^\top R \widetilde{K}^*_h + Q, \label{eq: P_tilde_star_def}
\end{align}
where $\widetilde{K}^*_h$ from \eqref{eq: K_argmin_def} can also be computed from
\begin{equation*}
    \widetilde{K}^*_h = (R + B^\top \widetilde{P}_{h+1} B)^{-1} B^\top \widetilde{P}_{h+1} A.
\end{equation*}
Finally, we have the Riccati difference equation (RDE):
\begin{align}
    P^*_{h} &= (A - B K^*_h)^\top P^*_{h+1} (A - B K^*_h) + (K^*_h)^\top R K^*_h + Q, \label{eq: RDE}
\end{align}
where $P^*_N = Q_N$ and $K^*_h$ from \eqref{eq: K_argmin_def} can also be computed from
\begin{equation}
    K^*_h = (R + B^\top P^*_{h+1} B)^{-1} B^\top P^*_{h+1} A. \label{eq: RDE_K}
\end{equation}
As a result, it is easy to follow that
\begin{align}
    \E_{x_0 \sim \D} \left[ x_0^\top \widetilde{P}_h x_0 \right] &= \widetilde{J}_h (\widetilde{K}_h), \\
    \E_{x_0 \sim \D} \left[ x_0^\top \widetilde{P}^*_h x_0 \right] &= \widetilde{J}_h (\widetilde{K}^*_h), \quad \text{and} \\
    \E_{x_0 \sim \D} \left[ x_0^\top P^*_h x_0 \right] &= J_h (K^*_h).
\end{align}
We also define the Riccati operator
\begin{equation}
    \mathcal{R}(P) := Q + A^\top (P - PB(R + B^\top P B)^{-1} B^\top P) A, \label{eq: def_Riccati_operator}
\end{equation}
so that $\widetilde{P}^*_h$ and $P^*_h$ can also be shown as
\begin{align}
    \widetilde{P}^*_h &= \mathcal{R}(\widetilde{P}_{h+1}) \label{eq: Riccati_operator_on_P_tilde_star} \\
    P^*_h &= \mathcal{R}(P^*_{h+1}), \label{eq: Riccati_operator_on_P_star}
\end{align}
after replacing $\widetilde{K}^*_h$ and $K^*_h$ in~\eqref{eq: P_tilde_star_def} and~\eqref{eq: RDE} respectively.

We next record the nonexpansiveness of the Riccati operator $\mathcal{R} (\cdot)$ under $\delta (\cdot, \cdot)$. Existing contraction results establish this property when $A$ is nonsingular~\cite{PB:93}.
% ; see also~\cite[Theorem~4.4]{HL-YL:08}. 
The following lemma extends the result to arbitrary $A$.

\begin{lemma} \label{lem: delta_Riccati_contraction}
    Consider the Riccati operator $\mathcal{R}(\cdot)$ defined in~\eqref{eq: def_Riccati_operator}. For any symmetric positive definite matrices \( X, Y \in \mathbb{R}^{n \times n} \), we have
    \begin{equation*}
        \delta(\mathcal{R}(X),\mathcal{R}(Y)) \leq \delta(X,Y).
    \end{equation*}
\end{lemma}
\begin{proof}
First, we make the dependence of the Riccati operator on $A$ explicit and write
\[
    \mathcal{R}_A(P)
    :=
    Q + A^\top \left(P-PB(R+B^\top P B)^{-1}B^\top P\right)A.
\]
If $A$ is nonsingular, the result follows directly from~\cite[Theorem~1.7]{PB:93}. Now let $A$ be arbitrary. Since the set of nonsingular matrices is dense in $\mathbb{R}^{n\times n}$, there exists a sequence of nonsingular matrices $\{A_j\}$ such that $A_j\to A$. For each $j$, the result for the nonsingular case gives
\begin{equation}
    \delta\big(\mathcal{R}_{A_j}(X),\mathcal{R}_{A_j}(Y)\big)
    \leq
    \delta(X,Y).
    \label{eq:RDE_contraction_sequence_prelim}
\end{equation}
Fix $P\succ0$, and define
\[
    S(P)
    :=
    P-PB(R+B^\top P B)^{-1}B^\top P.
\]
By the Woodbury identity,
\[
    S(P)
    =
    \left(P^{-1}+BR^{-1}B^\top\right)^{-1}
    \succ0.
\]
Hence
\[
    \mathcal{R}_{A_j}(P)
    =
    Q+A_j^\top S(P)A_j
    \succeq Q\succ0,
\]
and similarly $\mathcal{R}_A(P)\succ0$. Moreover, for fixed $P$, the map
\[
    A\mapsto Q+A^\top S(P)A
\]
is continuous. Thus, by continuity,
\[
    \lim_{j\to\infty}\mathcal{R}_{A_j}(P)
    =
    \mathcal{R}_A(P).
\]
In particular, as $j \to \infty$,
\[
    \mathcal{R}_{A_j}(X)\to\mathcal{R}_A(X),
    \qquad
    \mathcal{R}_{A_j}(Y)\to\mathcal{R}_A(Y).
\]
By the joint continuity of the Riemannian distance $\delta(\cdot,\cdot)$ on the cone of positive definite matrices,
\[
    \delta\big(\mathcal{R}_A(X),\mathcal{R}_A(Y)\big)
    =
    \lim_{j\to\infty}
    \delta\big(\mathcal{R}_{A_j}(X),\mathcal{R}_{A_j}(Y)\big).
\]
Taking $j\to\infty$ in~\eqref{eq:RDE_contraction_sequence_prelim} yields
\[
    \delta\big(\mathcal{R}_A(X),\mathcal{R}_A(Y)\big)
    \leq
    \delta(X,Y),
\]
which proves the result.
\end{proof}

Having introduced all the necessary definitions, we now turn our attention to the outer loop. 

\section{The Outer Loop (Receding-Horizon Policy Gradient)}
It has been demonstrated that the solution to the RDE~\eqref{eq: RDE} converges to the stabilizing solution of the ARE~\eqref{eq: ARE} exponentially~\cite{BH-AHS-TK:99}. As a result, $\{K^*_t\}_{t \in \{ N-1, \dotsc, 1, 0\}}$ in~\eqref{eq: RDE_K} also converges to $K^*$ as $N$ increases. In particular, we recall the following result from~\cite[Theorem~1]{XZ-TB:23}.
\begin{theorem}\label{thm: RDE converges to ARE}
     Let $A^*_K := A - BK^*$, and define
    \begin{equation}
        N_0 = \frac{1}{2} \frac{\log{\left(\frac{\| Q_N - P^* \|_* \kappa_{P^*} \|A^*_K \| \|B\|}{\eps \lambda_{\min} (R)}\right)}}{\log{\left(\frac{1}{\| A^*_K \|_*}\right)}},
    \end{equation}
    where $Q_N \succeq P^*$.
    Then it holds that $\| A^*_K \|_* < 1$ and for all $N \geq N_0$, the control policy $K^*_0$ computed by~\eqref{eq: RDE_K} is stabilizing and satisfies $\| K^*_0 - K^* \| \leq \eps$ for any $\eps>0$.
\end{theorem}

The proof of Theorem~\ref{thm: RDE converges to ARE} is provided in Appendix~\ref{app: RDE converges to ARE} for completeness (and to account for some minor change in notation). We also note that this theorem relies on a minor inherent assumption that \( Q_N \) satisfies \( Q_N \succeq P^* \). A full discussion of this assumption is provided in Remark~\ref{rem: Q_N inherent_assumption} in Appendix~\ref{app: RDE converges to ARE}. In addition, let
\[
P_{\max}:=\|P^*\|+\kappa_{P^*}\|Q_N-P^*\|_*.
\]
Then, as later shown in~\eqref{eq: P_star_uniform_bound}, we have that
\[
\|P^*_t\| \leq P_{\max}, 
\]
for all $t\in\{0,1,\dotsc,N\}$.

Before presenting the formal algorithm, we provide an intuitive overview of its structure and the information it uses.

\begin{figure}[t]
\centering
\resizebox{\textwidth}{!}{%
\begin{tikzpicture}[
  >=Stealth,
  node distance=10mm and 8mm,
  box/.style={draw=black, rounded corners, thick, align=center, inner sep=3pt, minimum height=13mm, text width=3.3cm, text=black},
  lab/.style={font=\normalsize, inner sep=1pt, text=black},
  arr/.style={->, thick, draw=black}
]

% --- Single-row boxes (outer flow) ---
\node[box, text width=4.2cm] (freeze) {Freeze downstream policies\\[1mm]\footnotesize fix $\widetilde{K}_{h+1},\dotsc, \widetilde{K}_{N-1}$};
\node[box, right=5mm of freeze, text width=4.4cm] (sur) {Build surrogate cost $\widetilde{J}_h(\cdot)$\\[1mm]\footnotesize function of the policy at $h$, with policies at $[h+1,N-1]$ fixed};

% --- start box ---
\node[box, left= of freeze, text width=1.5cm] (initialize) {\textbf{Initialize}};

% Inner loop: split into two boxes (shorter labels)
\node[box, right=of sur, text width=3.2cm, minimum height=14mm] (innerG) {Estimate gradient\\[1mm]\footnotesize $\widehat{\nabla}\widetilde{J}_h(K_{h,i})$};
\node[box, right=7mm of innerG, text width=4.2cm, minimum height=14mm] (innerU) {Update\\[1mm]\footnotesize $K_{h,i+1} \!\leftarrow\! K_{h,i}-\alpha_{h,i} \widehat{\nabla}\widetilde{J}_h (K_{h,i})$};

\node[box, right=of innerU, text width=2.6cm] (outK) {Output\\[1mm]\footnotesize $\widetilde{K}_h \gets K_{h,T_h}$};
% \node[box, right=of outK, text width=4.3cm] (shift) {Proceed to stage $h-1$\\[1mm]\footnotesize incorporate $\widetilde{K}_h$ into $\widetilde{J}_{h-1}$};

% --- Inner loop enclosure (taller from below) ---
\coordinate (belowInner) at ($(innerU.south)-(0,7mm)$);
\node[
  draw=black, dashed, rounded corners,
  fit=(innerG)(innerU)(belowInner),
  inner ysep=8pt,
  inner xsep=5pt
] (innerbox) {};

% --- Forward flow arrows ---
\draw[arr] (initialize.east) -- (freeze.west);
\draw[arr] (freeze) -- (sur);
\draw[arr] (sur) -- (innerbox);
\draw[arr] (innerG) -- (innerU);
% Inner loop back arrow (below to keep top clean)
\draw[arr, bend left=14] (innerU.south) to (innerG.south);

\draw[arr] (innerbox) -- (outK);
% \draw[arr] (outK) -- (shift);

\node[lab, above=0.5mm of innerbox.south] {\textbf{inner loop over } $i=0,1,\dots,T_h-1$};

% --- Outer loop back arrow + label ---
\coordinate (top) at ($(freeze.north)!0.5!(outK.north)+(0,1.3)$);
% \draw[arr, bend right=10] (shift.north) to (freeze.north);
\draw[arr, bend right=15] (outK.north) to (freeze.north);
\node[lab, above=12pt] at (top) {\textbf{outer loop (receding horizon) over $h$: $N-1,\dots,1,0$}};

\end{tikzpicture}
}% end resizebox
\caption{A brief overview of the nested-loop structure of Algorithm~\ref{alg:RHPG-RL}.}
\label{fig:algorithm_overview}
\end{figure}

Fig.~\ref{fig:algorithm_overview} highlights the two nested loops: an outer loop stepping backward in $h$ (from $N-1$ to $0$) and an inner loop, proceeding forward in $i$ (from $0$ to $T_h-1$), iteratively refining the policy at each stage. At each $h$, the algorithm (i) freezes downstream policies, (ii) builds the surrogate cost $\widetilde J_h$, (iii) runs an inner loop of gradient-based updates with some step-size $\alpha_{h,i}$ until approximate convergence, and (iv) outputs $\widetilde K_h$, which is incorporated into the surrogate cost for step $h-1$. 

We formalize this procedure in Algorithm~\ref{alg:RHPG-RL} and briefly state what information it uses in practice:

\begin{algorithm}[t]
\caption{Receding-Horizon Policy Gradient with RL-Inspired Gradient Estimator (RHPG-RL)}
\label{alg:RHPG-RL}
\begin{algorithmic}[1]
\REQUIRE Horizon $N$, max iterations $\{T_h\}$, stepsizes $\{\alpha_{h,i}\}$, variance parameter $\sigma^2$
\FOR{$h = N-1$ to $0$}
    \STATE Initialize $K_{h, 0}$ arbitrarily (e.g., $K_{h, 0} \leftarrow 0$).
    \FOR{$i = 0$ to $T_h-1$}
        \STATE Sample $\eta_{h,i} \sim \N (0, I_m)$ and simulate a new trajectory of the system~\eqref{eq: LTI_discrete_time} with 
        \[
        u_t = \begin{cases}
            - K_{h,i} x_{0} + \sigma \eta_{h,i} & \text{: }t=0, \\
            -\widetilde{K}_{h+t} x_t & \text{: } t \ge 1,
        \end{cases}
        \]
        and $x_0 \sim \mathcal{D}$.
        % and $u_{h'} = - K_{h',T_{h'}} x_{h'}$ for any $h' \in \{h+1,\dotsc,N-1\}$. Note that this $K_{h',T_{h'}}$ is essentially our $\widetilde{K}_{h'}$.
        \STATE Compute $Q_h(x_0, u_{0})$ for said trajectory.
        \STATE Compute the gradient estimate
        \[\widehat{\nabla} \widetilde{J}_{h} (K_{h, i}) \leftarrow - \frac{1}{\sigma} Q_h(x_0, u_0) \eta_{h,i} x_0^\top.\]
        \vspace{-2mm}
        \STATE Update $K_{h, i+1} \leftarrow K_{h, i} - \alpha_{h,i} \widehat{\nabla} \widetilde{J}_h(K_{h, i})$.
    \ENDFOR
    \STATE $\widetilde{K}_h \gets K_{h,T_{h}}$. 
    \STATE Incorporate $\widetilde{K}_h$ into the surrogate cost function for the next step, i.e., $\widetilde{J}_{h-1}(\cdot)$.
\ENDFOR
\RETURN $\widetilde{K}_0$.
\end{algorithmic}
\end{algorithm}

\textbf{Information available to the algorithm.}
At each stage $h$, the updates require neither knowledge of the system matrices $(A,B)$ nor full trajectory information. Instead, the algorithm draws $\eta \sim \mathcal{N}(0,I)$, applies a perturbed action $u_0 = -K_{h,i} x_0 + \sigma \eta$ where $x_0 \sim \D$, and then rolls out with the frozen downstream policies. It then queries a cost oracle for the finite-horizon cost-to-go $Q_h(x_0,u_0)$ of length $N-h$, which is used to form the gradient estimate for the inner updates. A detailed construction of $Q_h$ is deferred to Remark~\ref{rem: Q-calculation}.

Note that in this section, we focus on the outer loop of Algorithm~\ref{alg:RHPG-RL}, analyzing the requirements it imposes on the convergence of the policy gradient method employed in the inner loop. The details of the policy gradient method will be discussed in the next section.

Before we move on to the next result, we 
% select an arbitrary scalar $a>0$ and 
define the following constants:
\begin{align*}
    &a := \frac{\sigma_{\min} (Q)}{2}, \quad \varphi := \max_{t \in \{0,1,\dotsc,N-1\}} \| A - B K^*_t \|, \\
    &C_1 := \frac{\varphi \| B \|}{\lambda_{\min} (R)}, \quad C_2 := 2 \varphi \|A\| \left( 1 + \frac{(P_{\max}+a) \|B\|^2}{\lambda_{\min} (R)} \right), \quad C_3 := 2 \left( \| R \| + (P_{\max}+a) \| B \|^2 \right).
\end{align*}
Additionally, given a scalar $\eps >0$, we define:
\begin{equation}
\varsigma_{h,\eps} :=
\begin{cases} 
\displaystyle
\min\!\left\{
\begin{array}{l}
    \sqrt{\frac{a}{C_3} \frac{1}{N}},
    \sqrt{\frac{a}{C_3} \frac{a}{2 e N P_{\max}}}, \\[1mm]
    \sqrt{\frac{a}{C_3} \frac{\eps}{8 e N C_1 P_{\max} }},
    \sqrt{\frac{\eps}{4 C_1 C_3}} 
\end{array}
\right\} & h \geq 1, \\[2mm]
\displaystyle \frac{\eps}{4} & h = 0.
\end{cases}
\label{eq: def_varsigma_h}
\end{equation}

We now present a key result, Theorem~\ref{thm: LQR_DP_by_delta}, on the accumulation of errors. The result relies on a fundamentally different analysis from~\cite[Theorem~2]{XZ-TB:23} and yields an improved error propagation bound.

\begin{theorem}\longthmtitle{Main result: outer loop} \label{thm: LQR_DP_by_delta}
    Select 
    \begin{equation}
    N = \frac{1}{2} \cdot \frac{\log\left(\frac{2\|Q_N - P^*\|_*\cdot \kappa_{P^*} \cdot \|A^*_K\|\cdot \|B\|}{\eps \cdot \lambda_{\min}(R)}\right)}{\log\left(\frac{1}{\|A^*_K\|_*}\right)} + 1, \label{eq: N_value_choice}
    \end{equation}
    where 
    $Q_N \succeq P^*$. 
    Now assume that, for some \(\eps > 0\), there exists a sequence of policies \(\{\widetilde{K}_h\}_{h \in \{N-1, \dotsc, 0\}}\) such that for all \( h \in \{N-1, \dotsc, 0\} \),
    \[
    \| \widetilde{K}_h - \widetilde{K}^*_h \| \leq \varsigma_{h,\eps},
    \]
    where \(\widetilde{K}^*_h\) is the optimal policy for the Linear Quadratic Regulator (LQR) problem from step \(h\) to \(N\), incorporating errors from previous iterations of Algorithm~\ref{alg:RHPG-RL}.
    Then, the proposed algorithm outputs a control policy $\widetilde{K}_0$ that satisfies $\| \widetilde{K}_0 - K^* \| \leq \eps$. Moreover, for a sufficiently small $\eps$ such that $\eps < \frac{1 - \|A - BK^* \|_*}{\kappa_{P^*}^{1/2} \ \|B\|}$, $\widetilde{K}_0$ is stabilizing.
    % Furthermore, if $\eps$ is sufficiently small such that 
    % \[
    % \eps < \frac{1 - \| A - B K^* \|_*}{\| B \|},
    % \]
    % then $\widetilde{K}_0$ is stabilizing.
\end{theorem}

The condition on the sequence $\{\widetilde{K}_h\}_{h \in \{N-1, \dotsc, 0\}}$ in Theorem~\ref{thm: LQR_DP_by_delta} is satisfied by the inner-loop convergence analysis developed in Section~\ref{sec: inner loop}; in particular, Theorem~\ref{thm: policy_gradient} provides the required accuracy at each stage $h$.
Moreover, we note that the system-dependent quantities appearing in the bounds above need not be known exactly; suitable bounds on these quantities are sufficient for the corresponding parameter choices.

The proof of Theorem~\ref{thm: LQR_DP_by_delta} is presented in Appendix~\ref{app: thm_LQR_DP_by_Delta}. A key component of our analysis is the contraction of the Riccati operator under the Riemannian distance, as established in Lemma~\ref{lem: delta_Riccati_contraction}. This allows us to demonstrate that the accumulated error remains linear in $N$, in contrast to the exponential growth in~\cite[Theorem 2]{XZ-TB:23}.

To clarify this difference, we next revisit~\cite[Theorem~2]{XZ-TB:23} and state the corresponding result after correcting the error in the analysis of how the outer-loop approximation errors propagate.
For this reason, we provide a complete proof in Appendix~\ref{app: thm_LQR_DP}.
\begin{theorem}\longthmtitle{Prior result: outer loop} \label{thm: LQR_DP}
Choose
\begin{equation}
N = \frac{1}{2} \cdot \frac{\log\left(\frac{2\|Q_N - P^*\|_*\cdot \kappa_{P^*} \cdot \|A^*_K\|\cdot \|B\|}{\eps \cdot \lambda_{\min}(R)}\right)}{\log\left(\frac{1}{\|A^*_K\|_*}\right)} + 1, \label{eq: N value choice}
\end{equation}
where $Q_N \succeq P^*$. Now assume that, for some \(\eps > 0\), there exists a sequence of policies \(\{\widetilde{K}_h\}_{h \in \{N-1, \dotsc, 0\}}\) such that
\begin{equation}
\| \widetilde{K}_h - \widetilde{K}^*_h \| \leq
\begin{cases}
    \min\left\{
    \begin{array}{l}
        \sqrt{\frac{a}{C_3}},
        \sqrt{\frac{a}{C_2^{h-2}C_3}}, \\
        \frac{1}{2} \sqrt{\frac{\eps}{C_1C_2^{h-2}C_3}}
    \end{array} \right\} & h \geq 2, \\[8pt]
    \min\left\{
    \begin{array}{l}
        \sqrt{\frac{a}{C_3}},
        \frac{1}{2} \sqrt{\frac{\eps}{C_1C_3}}
    \end{array} \right\} & h = 1, \\[8pt]
    \frac{\eps}{4} & h = 0,
\end{cases}
\label{eq: refined_ineq_h_case}
\end{equation}
where \(\widetilde{K}^*_h\) is the optimal policy for the Linear Quadratic Regulator (LQR) problem from step \(h\) to \(N\), incorporating errors from previous iterations of Algorithm~\ref{alg:RHPG-RL}. Then, the RHPG algorithm outputs a control policy $\widetilde{K}_0$ that satisfies $\| \widetilde{K}_0 - K^* \| \leq \eps$. 
% Furthermore, if $\eps$ is sufficiently small such that 
% \[
% \eps < \frac{1 - \| A - B K^* \|_*}{\| B \|},
% \]
% then $\widetilde{K}_0$ is stabilizing.
\end{theorem}

As previously mentioned, Theorem~\ref{thm: LQR_DP_by_delta} yields a tighter error-propagation bound than Theorem~\ref{thm: LQR_DP}. The difference stems from the exponent of the constant $C_2$ in~\eqref{eq: refined_ineq_h_case}, which is discussed in detail in Appendix~\ref{app: thm_LQR_DP}.

\section{The Inner Loop and Policy Gradient}\label{sec: inner loop}
In this section, we focus on the inner loop of Algorithm~\ref{alg:RHPG-RL}, on which we will implement our proposed policy gradient method. 

We seek a way to estimate the gradient of the cost function $\widetilde{J}_h (K)$ with respect to $K$. To remedy, we propose:
\begin{align}
    \widehat{\nabla} \widetilde{J}_h (K) := - \frac{1}{\sigma^2} Q_h(x_0, u_0) (u_0 + K x_0) x_0^\top, \label{eq: grad_est_prelim}
\end{align}
where $x_0$ is sampled from $\mathcal{D}$, and then $u_0$ is chosen randomly from the Gaussian distribution $\mathcal{N} (- K x_0, \sigma^2 I_m)$ for some $\sigma > 0$. Moreover, we  rewrite $u_0 \sim \mathcal{N} (- K x_0, \sigma^2 I_m)$ as
\begin{equation}
    u_0 = - K x_0 + \sigma \eta, \label{eq: random_input_rewrite}
\end{equation}
where $\eta \sim \N (0, I_m)$. Substituting \eqref{eq: random_input_rewrite} in \eqref{eq: grad_est_prelim} yields
\begin{align}
    \widehat{\nabla} \widetilde{J}_h (K) = - \frac{1}{\sigma} Q_h(x_0, - K x_0 + \sigma \eta) \eta x_0^\top. \label{eq: grad_est_practical_form}
\end{align}
This expression corresponds to the gradient estimate used in Algorithm~\ref{alg:RHPG-RL}, as described in its formulation.

We now provide the following remark on the computation of $Q_h (x_0, u_0)$ for this gradient estimate:
\begin{remark} \label{rem: Q-calculation}
    For any fixed $h$, the $Q$-function in~\eqref{eq: grad_est_practical_form} represents the finite-horizon cost-to-go of horizon length $N-h$, i.e., $N-h$ stage costs followed by a terminal cost (indices in the notation can be perceived as shifts). To clarify, let us use the following notation
    \begin{align*}
    % \tilde{x}_t &:= x_{h+t} \quad \text{for all }t \in \{0,1,\dotsc,N-h\} \\
    u_t &:= \begin{cases}
        -K x_0 + \sigma \eta, & \text{: } t = 0, \\
        -\widetilde{K}_{h+t} \ x_t, & \text{: } t \ge 1,
    \end{cases}
    \end{align*}
    where $(x_t,u_t)$ form the trajectory of the system, i.e., the dynamics follow
    \[
    x_{t+1} = A x_t + B u_t,
    \]
    and $x_0 \sim \D$. Now using the quadratic stage cost 
    \[
    c^{(h)}_t := \begin{cases} x_t^\top Q_N x_t & \text{: } t = N-h, \\
    x_t^\top Q x_t + u_t^\top R u_t & \text{: otherwise},
    \end{cases}
    \]
    we can write
    \[
    Q_h (x_0, u_0) = \sum_{t=0}^{N-h} c^{(h)}_t.
    \]
    This gives the rollout representation of $Q_h(x_0,u_0)$ used in our gradient estimator. \oprocend
\end{remark}

Thus, the expression in~\eqref{eq: grad_est_practical_form} is a one-point gradient estimator that requires only a single observation of the cost $Q_h(x_0,u_0)$ at each iteration, as opposed to a two-point estimator, which requires two such observations for different policies under the exact same realization of the randomness $x_0$.

\textbf{Roadmap of the section.}
The core of our analysis in Section~\ref{sec: inner loop} relies on three key results: 
\textbf{(i)} unbiasedness of the gradient estimator (Proposition~\ref{prop: gradient estimate expectation}), 
\textbf{(ii)} a bound on the second-moment of the Frobenius norm of the estimate (Lemma~\ref{lem: grad_est_bounds}), and 
\textbf{(iii)} regularity properties of the stage cost $\widetilde{J}_h$ (Lemma~\ref{lem: cost function properties}).
Together, these yield a one-step descent recursion (Lemma~\ref{lem: policy gradient recursive}),  which leads to a probabilistic guarantee on the behavior of the optimality-gap (Proposition~\ref{prop: lqr_policy}) and culminates in the main inner-loop convergence theorem (Theorem~\ref{thm: policy_gradient}). Fig.~\ref{fig:sec4_roadmap} summarizes these logical dependencies, with the three bolded results serving as the main pillars of the argument.

\begin{figure}[t]
\centering
\begin{tikzpicture}[
  node distance=7mm and 12mm,
  >=Stealth,
  box/.style={draw=black, rounded corners, align=center, font=\small, inner sep=3pt, text width=5.3cm, text=black},
  emph/.style={draw=black, rounded corners, align=center, font=\small\bfseries, inner sep=3pt, text width=5.5cm, text=black},
  arr/.style={->, thick, draw=black}
]

% --- Column 1: three "ingredients" stacked ---
\node[emph, text width=3.6cm] (unb) {Proposition~\ref{prop: gradient estimate expectation}\\(unbiasedness of {\boldmath$\widehat{\nabla}\widetilde J_h$})};
\node[emph, below=of unb, text width=3.6cm] (var) {Lemma~\ref{lem: grad_est_bounds}\\(second-moment bound of {\boldmath$\widehat{\nabla}\widetilde J_h$})};
\node[emph, below=of var, text width=3.6cm] (geom) {Lemma~\ref{lem: cost function properties}\\(regularity properties of {\boldmath$\widetilde{J}_h$})};

% --- Column 2: recursion (aligned to middle item in col 1) ---
\node[box, right=8mm of var, text width=2.5cm] (rec) {Lemma~\ref{lem: policy gradient recursive}\\(one-step descent recursion)};

% --- Column 3: finite-time ---
\node[box, right=8mm of rec, text width=3.5cm] (finite) {Proposition~\ref{prop: lqr_policy}\\(probabilistic optimality-gap guarantee)};

% --- Column 4: main theorem ---
\node[box, right=8mm of finite, text width=3.5cm] (main) {Theorem~\ref{thm: policy_gradient}\\(main inner-loop convergence guarantee)};

% --- Arrows from ingredients to recursion ---
\draw[arr] (unb.east) -- (rec.west);
\draw[arr] (var.east) -- (rec.west);
\draw[arr] (geom.east) -- (rec.west);

% --- Arrows through the chain ---
\draw[arr] (rec.east) -- (finite.west);
\draw[arr] (finite.east) -- (main.west);

\end{tikzpicture}
\caption{Roadmap of technical results in Section~\ref{sec: inner loop}.}
\label{fig:sec4_roadmap}
\end{figure}

We start by showing that the estimator is unbiased with respect to true gradient:

\begin{proposition}
    \label{prop: gradient estimate expectation}
    Suppose $x_0 \sim \D$ and $u_{0}$ is drawn from $\N (- K x_0, \sigma^2 I_m)$ as before. Then for any $K$, we have that
    \begin{equation}
        \E [\widehat{\nabla} \widetilde{J}_h (K)] = \nabla \widetilde{J}_h (K).
    \end{equation}
\end{proposition} 

\begin{proof}
    Following \eqref{eq: grad_est_practical_form}, 
    \begin{align}
        \E [\widehat{\nabla} \widetilde{J}_h (K)] &= \E_{x_0} \left[ \E_{\eta} \left[ \widehat{\nabla} \widetilde{J}_h (K) \big| x_0 \right] \right] \notag \\
        &\overset{\mathrm{(i)}}{=} \E_{x_0} \left[ - \frac{1}{\sigma^2} \E_{\eta} \left[ Q_h (x_{0}, - K x_{0} + \sigma \eta) (\sigma \eta) \big| x_0 \right] x_0^\top \right] \cr
        &\overset{\mathrm{(ii)}}{=} \E_{x_0} \left[ \E_{\eta} \left[ - \nabla_{u} Q_h (x_{0}, u) \bigg|_{u = -K x_{0} + \sigma \eta} \big| x_0 \right] x_{0}^\top \right], \label{eq: policy gradient proof before Stein}
    \end{align}
    where (i) follows from $x_{0}^\top$ being determined when given $x_0$, and (ii) from Stein's lemma~\cite{CMS:81}. Using~\eqref{eq: Q_func_def}, we compute 
    \begin{align*}
        \nabla_{u} Q_h (x_{0}, u) &= \nabla_{u} \left( x_{0}^\top Q x_{0} + u^\top R u + (A x_{0} + B u)^\top \widetilde{P}_{h+1} (A x_{0} + B u) \right) \\
        &= 2 R u + 2 B^\top \widetilde{P}_{h+1} B u + 2 B^\top \widetilde{P}_{h+1} A x_{0},
    \end{align*}
    which evaluated at $u = -K x_{0} + \sigma \eta$ yields
    \begin{align*}
        \nabla_{u} Q_h (x_{0}, u) \bigg|_{u = -K x_{0} + \sigma \eta} = 2 \left( (R + B^\top \widetilde{P}_{h+1} B)(-K x_{0} + \sigma \eta) + B^\top \widetilde{P}_{h+1} A x_{0} \right).
    \end{align*}
    Substituting in \eqref{eq: policy gradient proof before Stein}, we obtain
    \begin{align*}
        \E [\widehat{\nabla} \widetilde{J}_h (K)] &= \E_{x_0 \sim \mathcal{D}} \left[ 2 \left( (R + B^\top \widetilde{P}_{h+1} B)K - B^\top \widetilde{P}_{h+1} A \right) x_{0} x_{0}^\top \right] \\
        &= 2 \left( (R + B^\top \widetilde{P}_{h+1} B)K - B^\top \widetilde{P}_{h+1} A \right) \Sigma_0 \\
        &\overset{\mathrm{(i)}}{=} \nabla \widetilde{J}_h (K),
    \end{align*}
    where (i) follows from~\eqref{eq: actual_gradient}.
\end{proof}

Similar to~\cite{XZ-TB:23}, we define the following sets regarding the inner loop of the algorithm for each $h \in \{0,1,\dotsc,N-1\}$:
\begin{equation}
    \G_h := \{ K_h | \widetilde{J}_h (K_h) - \widetilde{J}_h(\widetilde{K}^*_h) \leq 10 \zeta^{-1} \widetilde{J}_h (K_{h,0}) \}, \label{eq: G_lqr_h_def}
\end{equation}
for some arbitrary $\zeta \in (0,1)$. We also define the following constant:
\begin{align*}
    \widetilde{C}_{h} := \frac{10 \zeta^{-1} \widetilde{J}_h (K_{h,0}) + \widetilde{J}_h(\widetilde{K}^*_h)}{\sigma_{\min} (\Sigma_0) \sigma_{\min} (R)}.
\end{align*} 
We next provide a second-moment bound on the size of the estimator:

\begin{lemma}\label{lem: grad_est_bounds}
    Suppose $\| \widetilde{P}_{h+1} - P^*_{h+1} \| \leq a$.
    % for some fixed scalar $a>0$. 
    Then for any $K \in \G_h$, we have that
    \begin{equation}
        \E\left[ \| \widehat{\nabla} \widetilde{J}_h (K) \|_F^2 \right] \leq \xi_{h,3},
    \end{equation}
    where $\xi_{h,1}, \xi_{h,2}, \xi_{h,3} \in \real$ are given by
    \begin{align}
        \xi_{h,1} &:= \Big( \| Q \| + 2 \| R \| \widetilde{C}_{h}^2 + 2 (P_{\max}+a) (\|A\|^2 + 2 \|B\|^2 \widetilde{C}_{h}) \Big)  C_m^{3/2}, \\
        \xi_{h,2} &:= 2 \left( \| R \| + 2 (P_{\max}+a) \| B \|^2 \right) C_m^{1/2}, \\
        \xi_{h,3} &:= \frac{1}{\sigma^2} \xi_{h,1}^2 m + 2 \xi_{h,1} \xi_{h,2} m (m+2) + \sigma^2 \xi_{h,2}^2 m (m+2) (m+4). \label{eq: xi_3 definition}
    \end{align}
\end{lemma}
\begin{proof}
    Using the Formulation of $\widehat{\nabla} \widetilde{J}_h (K)$ derived in \eqref{eq: grad_est_practical_form}, we have
    \begin{align}
        \| \widehat{\nabla} \widetilde{J}_h (K) \|_F 
        % &= \| \frac{1}{\sigma} Q_h(x_0, - K x_0 + \sigma \eta) \eta x_0^\top \|_F \cr
        &\leq \frac{1}{\sigma} Q_h(x_0, - K x_0 + \sigma \eta) \| \eta \| \| x_0 \|. \label{eq: lem_grad_est_bounds_1}
    \end{align}
    Before we continue, we provide the following bound:    \begin{sublemma}\label{sublem: K_size_bound}
        Suppose $K \in \G_h$. Then it holds that
        \begin{equation}
            \| K \|^2_F \leq \widetilde{C}_{h}.
        \end{equation}
    \end{sublemma}
    \textit{Proof of Sublemma~\ref{sublem: K_size_bound}:}
    Using \eqref{eq: J_h_tilde_definition}, we have
    \begin{align}
        \widetilde{J}_h (K) &\geq \E_{x_0 \sim \D} \left[ x_0^\top (Q+K^\top R K) x_0 \right] \cr
        % &= \E_{x_0 \sim \D} \left[ \tr \left( (Q+K^\top R K) x_0 x_0^\top \right) \right] \cr
        &= \tr \left( (Q+K^\top R K) \Sigma_0 \right) \cr
        % &\geq \sigma_{\min} (\Sigma_0) \tr (Q+K^\top R K) \cr
        % &\geq \sigma_{\min} (\Sigma_0) \tr (R K K^\top) \cr
        &\geq \sigma_{\min} (\Sigma_0) \sigma_{\min} (R) \| K \|^2_F. \label{eq: sublem_K_size_bound_1}
    \end{align}
    Rearranging \eqref{eq: sublem_K_size_bound_1} yields
    \begin{align*}
        \| K \|^2_F &\leq \frac{\widetilde{J}_h (K)}{\sigma_{\min} (\Sigma_0) \sigma_{\min} (R)} \\
        &\overset{\mathrm{(i)}}{\leq} \frac{10 \zeta^{-1} \widetilde{J}_h (K_{h,0}) + \widetilde{J}_h(\widetilde{K}^*_h)}{\sigma_{\min} (\Sigma_0) \sigma_{\min} (R)} = \widetilde{C}_{h},
    \end{align*}
    where (i) follows from the definition of the set $\G_h$ in \eqref{eq: G_lqr_h_def}. \proofend
    
    We now continue with the proof of Lemma~\ref{lem: grad_est_bounds}. Note that
    \begin{align*}
        Q_h(x_0, - K x_0 + \sigma \eta) = &x_0^\top Q x_0 + (- K x_0 + \sigma \eta)^\top R (- K x_0 + \sigma \eta) \cr
        &+ (A x_0 + B(- K x_0 + \sigma \eta))^\top \widetilde{P}_{h+1} (A x_0 + B(- K x_0 + \sigma \eta)) \cr
        \leq &\| Q \| C_m + \| R \| \| - K x_0 + \sigma \eta \|^2 + \| \widetilde{P}_{h+1} \| \| A x_0 + B(- K x_0 + \sigma \eta) \|^2. 
    \end{align*} 
    As a result,
    \begin{align}
        Q_h(x_0, - K x_0 + \sigma \eta) \leq &\| Q \| C_m + 2 \| R \| (\widetilde{C}_{h} C_m + \sigma^2 \| \eta \|^2) + 2 (P_{\max}+a) \|A\|^2 C_m \cr 
        & + 4 (P_{\max}+a)  \|B\|^2 (\widetilde{C}_{h} C_m + \sigma^2 \| \eta \|^2) ) \cr
        = &C_m( \|Q\| + 2 \|R\| \widetilde{C}_{h}) + 2 C_m (P_{\max}+a) ( \|A\|^2 + 2 \|B\|^2 \widetilde{C}_{h} )\cr
        &+ 2 \left( \|R\| + 2 (P_{\max}+a) \|B\|^2 \right) \sigma^2 \| \eta \|^2, \label{eq: lem_grad_est_bounds_2}
    \end{align}    
     where the inequality follows from Sublemma~\ref{sublem: K_size_bound} along with the fact that by the assumption, 
    \begin{align*}
    \| \widetilde{P}_{h+1} \| 
    % &= \| P^*_{h+1} + (\widetilde{P}_{h+1} - P^*_{h+1}) \| \cr
    \leq \| P^*_{h+1} \| + \| \widetilde{P}_{h+1} - P^*_{h+1} \| \leq P_{\max}+a.
    \end{align*}
    Combining \eqref{eq: lem_grad_est_bounds_1} with \eqref{eq: lem_grad_est_bounds_2} and \eqref{eq: noise_assumptions}, we obtain
    \begin{align}
        \| \widehat{\nabla} \widetilde{J}_h (K) \|_F &\leq \frac{1}{\sigma} \Big( \|Q\| + 2 \|R\| \widetilde{C}_{h} + 2 (P_{\max}+a) ( \|A\|^2 + 2 \|B\|^2 \widetilde{C}_{h} ) \Big) C_m^{3/2} \| \eta \| \cr
        &\quad+ 2 \left( \|R\| + 2 (P_{\max}+a) \|B\|^2 \right) \sigma C_m^{1/2} \| \eta \|^3 \cr
        &= \frac{1}{\sigma} \xi_{h,1} \| \eta \| + \sigma \xi_{h,2} \| \eta \|^3. \label{eq: grad_est_size_bound_before_LM}
    \end{align}
    Now note that using~\eqref{eq: grad_est_size_bound_before_LM}, we have
    \begin{align*}
        \| \widehat{\nabla} \widetilde{J}_h (K) \|^2_F &\leq \frac{1}{\sigma^2} \xi_{h,1}^2 \| \eta \|^2 + 2 \xi_{h,1} \xi_{h,2} \| \eta \|^4 + \sigma^2 \xi_{h,2}^2 \| \eta \|^6.
    \end{align*}
    Now since $\| \eta \| \sim \chi (m)$ whose moments are known, taking an expectation results in
    \begin{align*}
        \E \left[ \| \widehat{\nabla} \widetilde{J}_h (K) \|^2_F \right] &\leq \frac{1}{\sigma^2} \xi_{h,1}^2 \E [\| \eta \|^2] + 2 \xi_{h,1} \xi_{h,2} \E [\| \eta \|^4] + \sigma^2 \xi_{h,2}^2 \E [\| \eta \|^6] \\
        &= \frac{1}{\sigma^2} \xi_{h,1}^2 m + 2 \xi_{h,1} \xi_{h,2} m (m+2) + \sigma^2 \xi_{h,2}^2 m (m+2) (m+4) \\
        &= \xi_{h,3},
    \end{align*}
    concluding the proof.
\end{proof}

We next establish some useful properties of the cost function $\widetilde{J}_h (K)$ in the following lemma.

\begin{lemma}\label{lem: cost function properties}
For all $h \in \{0,1,\dotsc,N-1\}$, the function $\widetilde{J}_h$ is $\frac{\mu}{2}$-strongly convex, where
\[
\mu := 4 \sigma_{\min} (\Sigma_0) \sigma_{\min} (R),
\]
and in particular, for all $K \in \real^{m \times n}$,
\begin{equation}
    \| \nabla \widetilde{J}_h (K) \|_F^2 \geq \mu (\widetilde{J}_h (K) - \widetilde{J}_h (\widetilde{K}^*_h)),
\end{equation}
where $\widetilde{K}^*_h$ is the global minimizer of $\widetilde{J}_h$. Moreover, assuming that $\| \widetilde{P}_{h+1} - P^*_{h+1} \| \leq a$, we have that for all $K_1,K_2 \in \real^{m \times n}$,
\begin{equation}
    \| \nabla \widetilde{J}_h (K_2) - \nabla \widetilde{J}_h (K_1) \|_F \leq L \| K_2 - K_1 \|_F,
\end{equation}
where 
\[
L := C_3 \| \Sigma_0 \|. 
\]
\end{lemma}
\begin{proof}
We first prove the strong convexity as follows:
\begin{align*}
    \left\langle \nabla \widetilde{J}_h (K_2) - \nabla \widetilde{J}_h (K_1), K_2 - K_1 \right\rangle &= 2 \tr \left( \Sigma_0 (K_2 - K_1)^\top (R + B^\top \widetilde{P}_{h+1} B) (K_2 - K_1) \right) \\
    &\geq 2 \sigma_{\min} (\Sigma_0) \sigma_{\min} (R) \tr \left( (K_2 - K_1)^\top (K_2 - K_1) \right) \\
    &= \frac{\mu}{2} \| K_2 - K_1 \|_F^2.
\end{align*}
Note the the next inequality is an immediate 
consequence of the PL-inequality. Now we move on to the $L$-smoothness property:
\begin{align*}
    \| \nabla \widetilde{J}_h (K_2) - \nabla \widetilde{J}_h (K_1) \|_F &= \| 2 (R + B^\top \widetilde{P}_{h+1} B) (K_2 - K_1) \Sigma_0 \|_F \\
    % \leq &\| \Sigma_0 \| (2 \| R + B^\top \widetilde{P}_{h+1} B \|) \| K_2 - K_1 \|_F \\
    &\leq \| \Sigma_0 \| (2 \| R + (P_{\max}+a) B^\top B \|) \| K_2 - K_1 \|_F \\
    &= C_3 \| \Sigma_0 \| \| K_2 - K_1 \|_F \\
    &= L \| K_2 - K_1 \|_F,
\end{align*}
concluding the proof.
\end{proof}

Before introducing the next result, let us denote the optimality gap of iterate $i$ by
% of the algorithm at each $h \in \{0,1,\dotsc,N-1\}$:
\begin{equation}
    \Delta_{i} = \widetilde{J}_h (K_{h,i}) - \widetilde{J}_h (\widetilde{K}^*_h). \label{eq: delta_def}
\end{equation}
Moreover, let $\F_{i}$ denote the $\sigma$-algebra containing the randomness up to iteration $i$ of the inner loop of the algorithm for each $h \in \{0,1,\dotsc,N-1\}$ (including $K_{h,i}$ but not $\widehat{\nabla} \widetilde{J}_h (K_{h,i})$). We then define
\begin{equation}
    \tau := \min \left\{ i \ |  \ \Delta_{i} > 10 \zeta^{-1} \widetilde{J}_h(K_{h,0}) \right\}, \label{eq: tau_1 definition}
\end{equation}
which is a stopping time with respect to $\F_{i}$. Note that there is a slight abuse of notation, as $\Delta_i$, $\F_i,$ and $\tau$ may differ for each $h \in \{0,1,\dotsc,N-1\}$. Since each outer-loop step $h$ is analyzed independently, we use the same notation for simplicity.

Combining the unbiasedness of the estimate (Proposition~\ref{prop: gradient estimate expectation}), the second-moment bound on its size (Lemma~\ref{lem: grad_est_bounds}), and the regularity properties of $\widetilde{J}_h$ (Lemma~\ref{lem: cost function properties}), we now show the following descent recursion for the optimality gap:
\begin{lemma}
    \label{lem: policy gradient recursive}
    Suppose $\| \widetilde{P}_{h+1} - P^*_{h+1} \| \leq a$, and the update rule follows
    \begin{equation}
        K_{h,i+1} = K_{h,i} - \alpha_{h,i} \widehat{\nabla} \widetilde{J}_h (K_{h,i}), \label{eq: our_update}
    \end{equation}
    where $\alpha_{h,i} > 0$ is the step-size. Then for any $i \in \{0,1,2,\dotsc\}$,
    \begin{equation}
        \E[\Delta_{i+1} | \F_{i}] 1_{\tau > i} \leq \left( \left(1 - \mu \alpha_{h,i} \right) \Delta_i + \frac{L \alpha_{h,i}^2}{2} \xi_{h,3} \right) 1_{\tau > i}, \label{eq: policy gradient recursive}
    \end{equation}
    where $\Delta_{i}$ is defined in \eqref{eq: delta_def}.
\end{lemma}
\begin{proof}
    First, note that by $L$-smoothness, we have
    \begin{align*}
        \Delta_{i+1} - \Delta_{i} &= \widetilde{J}_h (K_{h,i+1}) - \widetilde{J}_h (K_{h,i}) \\
        &\leq \langle \nabla \widetilde{J}_h (K_{h,i}), K_{h,i+1} - K_{h,i} \rangle 
        + \frac{L}{2} \| K_{h,i+1} - K_{h,i} \|_F^2 \\
        &= -\alpha_{h,i} \langle \nabla \widetilde{J}_h (K_{h,i}), \widehat{\nabla} \widetilde{J}_h (K_{h,i}) \rangle + \frac{L \alpha_{h,i}^2}{2} \| \widehat{\nabla} \widetilde{J}_h (K_{h,i}) \|_F^2,
    \end{align*}
    which after multiplying by $1_{\tau > i}$ (which is determined by $\F_{i}$) and taking an expectation conditioned on $\F_{i}$ gives
    \begin{align}
        \E [\Delta_{i+1} - \Delta_i | \F_i] 1_{\tau > i} &\leq -\alpha_{h,i} \langle \nabla \widetilde{J}_h (K_{h,i}), \E [\widehat{\nabla} \widetilde{J}_h (K_{h,i}) | \F_i] \rangle 1_{\tau > i} + \frac{L \alpha_{h,i}^2}{2} \E [\| \widehat{\nabla} \widetilde{J}_h (K_{h,i}) \|_F^2 | \F_i] 1_{\tau > i} \cr
        &\overset{\mathrm{(i)}}{\leq} -\alpha_{h,i} \| \nabla \widetilde{J}_h (K_{h,i}) \|_F^2 1_{\tau > i} + \frac{L \alpha_{h,i}^2}{2} \xi_{h,3} 1_{\tau > i} \cr
        &\overset{\mathrm{(ii)}}{\leq} -\alpha_{h,i} \mu \Delta_i 1_{\tau > i} + \frac{L \alpha_{h,i}^2}{2} \xi_{h,3} 1_{\tau > i}, \label{eq: lem_recursive_prelim}
    \end{align}
    where (i) follows from Proposition~\ref{prop: gradient estimate expectation}, Lemma~\ref{lem: grad_est_bounds} along with the fact that the event $\{\tau > i\}$ implies $K_{h,i} \in \G_h$, and (ii) is due to Lemma~\ref{lem: cost function properties}. 
    
    Now after some rearranging on~\eqref{eq: lem_recursive_prelim} and noting that $\Delta_i$ is also determined by $\F_i$, we conclude that 
    \[
        \E [\Delta_{i+1} | \F_i] 1_{\tau > i} \leq \left( \left(1 - \mu \alpha_{h,i} \right) \Delta_i + \frac{L \alpha_{h,i}^2}{2} \xi_{h,3} \right) 1_{\tau > i},
    \]
    finishing the proof.
\end{proof}

We are now in a position to state a precise version of our main result for the inner loop.

\begin{theorem}\label{thm: policy_gradient}\longthmtitle{Main result: inner loop}
    Suppose $\| \widetilde{P}_{h+1} - P^*_{h+1} \| \leq a$. 
    % for some scalar $a>0$. 
    For any $h \in \{0,1,\dotsc,N-1\}$, if the step-size is chosen as
    \begin{equation}
        \alpha_{h,i} = \frac{2}{\mu} \frac{1}{i+\theta_h} \quad \text{for} \quad \theta_h = \max \{2,\frac{2 L \xi_{h,3}}{\mu^2 \widetilde{J}_h (K_{h,0})}\}, \label{eq: step-size value choice}
    \end{equation}
    then for a given error tolerance $\varsigma$, the iterate $K_{h,T_h}$ of the update rule~\eqref{eq: our_update}
    after
    \[
    T_h = \frac{40}{7 \mu \varsigma^2 \zeta} \theta_h \widetilde{J}_h (K_{h,0})
    \]
    steps satisfies
    \[
    \| K_{h,T_h} - \widetilde{K}^*_{h} \|_F \leq \varsigma,
    \]
    with a probability of at least $1 - \zeta$.
\end{theorem}

Before proving Theorem~\ref{thm: policy_gradient}, we first establish the following preliminary result.

\begin{lemma}\label{lem: lqr_pg_inductive}
    Under the parameter setup of Theorem~\ref{thm: policy_gradient}, we have that
    \[
    \E[\Delta_i 1_{\tau > i}] \leq \frac{\theta_h \widetilde{J}_h (K_{h,0})}{i+\theta_h},
    \]
    for all $i \in \{0,1,\dotsc,T_h\}$.
\end{lemma}

\begin{proof}
We prove this result by induction on $i$ as follows:

\textbf{Base case ($i = 0$):} 
\begin{align*}
    \Delta_0 1_{\tau > 0} \leq \Delta_0 \leq \widetilde{J}_h (K_{h,0}) = \frac{\theta_h \widetilde{J}_h (K_{h,0})}{0 + \theta_h},
\end{align*}
which after taking expectation proves the claim for $i = 0$.

\textbf{Inductive step:} Let $k \in \{0,1,\dotsc,T_h-1\}$ be fixed and assume that
\begin{equation}
    \E[\Delta_k 1_{\tau > k}] \leq \frac{\theta_h \widetilde{J}_h (K_{h,0})}{k+\theta_h} \label{eq: sublem__pg_hyp}
\end{equation}
holds (the inductive hypothesis). Observe that
\begin{align}
    \E[\Delta_{k+1} 1_{\tau > k+1}] &\overset{\mathrm{(i)}}{\leq} \E [\Delta_{k+1} 1_{\tau > k} ] \cr
    &= \E [\E[\Delta_{k+1} 1_{\tau > k} | \F_k]] \cr
    &\overset{\mathrm{(ii)}}{=} \E [\E[\Delta_{k+1} | \F_k] 1_{\tau > k}], \label{eq: sublem__pg_conditional_expectation_bound}
\end{align}
where (i) comes from $1_{\tau > k+1} \leq 1_{\tau > k}$ and (ii) from the fact that $1_{\tau > k}$ is determined by $\F_k$. By Lemma~\ref{lem: policy gradient recursive}, we have that
\begin{align}
    \E[\Delta_{k+1} | \F_k] 1_{\tau > k}
    &\leq \left( \left(1 - \mu \alpha_{h,k} \right) \Delta_k + \frac{L \alpha_{h,k}^2}{2} \xi_{h,3} \right) 1_{\tau > k} \cr
    &\overset{\mathrm{(i)}}{\le} \left( 1 - \frac{2}{k+\theta_h} \right) \Delta_k 1_{\tau > k} + \frac{2 L \xi_{h,3}}{\mu^2} \left( \frac{1}{k+\theta_h} \right)^2, \label{eq: induction prelim 1}
\end{align}
where (i) comes from replacing $\alpha_{h,k}$ with its value in Theorem~\ref{thm: policy_gradient} along with the fact that $1_{\tau > k} \leq 1$. Now taking an expectation on~\eqref{eq: induction prelim 1} and combining it with~\eqref{eq: sublem__pg_conditional_expectation_bound} yields
\begin{align}
    \E[\Delta_{k+1} 1_{\tau > k+1}] &\leq \left( 1 - \frac{2}{k+\theta_h} \right) \E[\Delta_k 1_{\tau > k}] + \frac{2 L \xi_{h,3}}{\mu^2} \left( \frac{1}{k+\theta_h} \right)^2 \cr
    &\overset{\mathrm{(i)}}{\leq} \left( 1 - \frac{2}{k+\theta_h} \right) \frac{\theta_h \widetilde{J}_h (K_{h,0})}{k+\theta_h} + \frac{2 L \xi_{h,3}}{\mu^2} \left( \frac{1}{k+\theta_h} \right)^2 \cr
    &= \left( 1 - \frac{1}{k+\theta_h} \right) \frac{\theta_h \widetilde{J}_h (K_{h,0})}{k+\theta_h} - \frac{1}{(k+\theta_h)^2} \left( \theta_h \widetilde{J}_h (K_{h,0}) - \frac{2 L \xi_{h,3}}{\mu^2} \right) \cr
    &\overset{\mathrm{(ii)}}{\leq} \frac{k+\theta_h-1}{(k+\theta_h)^2} \theta_h \widetilde{J}_h (K_{h,0}) \cr
    &\leq \frac{1}{k+\theta_h+1} \theta_h \widetilde{J}_h (K_{h,0}),
\end{align}
where (i) comes from the induction hypothesis~\eqref{eq: sublem__pg_hyp}, and (ii) from 
\[
\theta_h \widetilde{J}_h (K_{h,0}) - \frac{2 L \xi_{h,3}}{\mu^2} \geq 0,
\]
which is due to the choice of $\theta_h$ in Theorem~\ref{thm: policy_gradient}. This proves the claim for $k+1$, completing the inductive step. 
\end{proof}

We now use Lemma~\ref{lem: lqr_pg_inductive} to prove Proposition~\ref{prop: lqr_policy}, which will then be used to establish Theorem~\ref{thm: policy_gradient}.
\begin{proposition}\label{prop: lqr_policy}
    Under the parameter settings of Theorem~\ref{thm: policy_gradient}, we have that
    \[
    \E[\Delta_{T_h} 1_{\tau > T_h}] \leq \frac{7}{40} \mu \varsigma^2 \zeta.
    \]
    Moreover, the event $\{ \tau \leq T_h \}$ happens with probability of at most $\frac{3}{10} \zeta$.
\end{proposition}
\begin{proof}
    Using Lemma~\ref{lem: lqr_pg_inductive} along with the choice of $T_h$ in Theorem~\ref{thm: policy_gradient}, we have
    \[
    \E[\Delta_{T_h} 1_{\tau > T_h}] \leq \frac{\theta_h \widetilde{J}_h (K_{h,0})}{T_h + \theta_h} \leq \frac{\theta_h \widetilde{J}_h (K_{h,0})}{T_h} \leq \frac{7 \mu \varsigma^2 \zeta}{40},
    \]
    concluding the proof of the first claim of Proposition~\ref{prop: lqr_policy}. Moving on to the second claim, we start by introducing the stopped process
    \begin{equation}
        Y_{i} := \Delta_{i \wedge \tau} + \frac{4 L \xi_{h,3}}{\mu^2} \frac{1}{i+\theta_h}. \label{eq: y_t supermartingale def}
    \end{equation}
    We now show this process is a supermartingale. First, observe that
    \begin{align}
        \E[Y_{i+1} | \F_i] &= \E [\Delta_{i+1 \wedge \tau} | \F_i] + \frac{4 L \xi_{h,3}}{\mu^2} \frac{1}{i+\theta_h+1} \cr
        &= \E [\Delta_{i+1 \wedge \tau} (1_{\tau \leq i} + 1_{\tau > i}) | \F_i] + \frac{4 L \xi_{h,3}}{\mu^2} \frac{1}{i+\theta_h+1} \cr
        &= \E [\Delta_{i+1 \wedge \tau} 1_{\tau \leq i} | \F_i] + \E [\Delta_{i+1 \wedge \tau} | \F_i] 1_{\tau > i} + \frac{4 L \xi_{h,3}}{\mu^2} \frac{1}{i+\theta_h+1}. \label{eq: supermartingale_prelim_1}
    \end{align}
    Now note that for the first term of the right-hand side of~\eqref{eq: supermartingale_prelim_1}, it holds that
    \begin{align}
        \E [\Delta_{i+1 \wedge \tau} 1_{\tau \leq i} | \F_i] \overset{\mathrm{(i)}}{=} \E [\Delta_{i \wedge \tau} 1_{\tau \leq i} | \F_i] = \Delta_{i \wedge \tau} 1_{\tau \leq i}, \label{eq: supermartingale_prelim_2}
    \end{align}
    where (i) follows from the fact that under the event $\{ \tau \leq i \}$, we have $\Delta_{i+1 \wedge \tau} = \Delta_{i \wedge \tau}$. Moreover, for the second term of the right-hand side of~\eqref{eq: supermartingale_prelim_1}, we have that
    \begin{align}
        \E [\Delta_{i+1 \wedge \tau} | \F_i] 1_{\tau > i} &\overset{\mathrm{(i)}}{\leq} \left( 1 - \frac{2}{i+\theta_h} \right) \Delta_i 1_{\tau > i} + \frac{2 L \xi_{h,3}}{\mu^2} \left( \frac{1}{i+\theta_h} \right)^2 1_{\tau > i} \cr
        &\leq \Delta_i 1_{\tau > i} + \frac{2 L \xi_{h,3}}{\mu^2} \left( \frac{1}{i+\theta_h} \right)^2, \label{eq: supermartingale_prelim_3}
    \end{align}
    where (i) follows from Lemma~\ref{lem: policy gradient recursive}. Combining \eqref{eq: supermartingale_prelim_2} and \eqref{eq: supermartingale_prelim_3} with \eqref{eq: supermartingale_prelim_1}, we get
    \begin{align}
        \E[Y_{i+1} | \F_i] &\leq \Delta_{i \wedge \tau} 1_{\tau \leq i} + \Delta_i 1_{\tau > i} + \frac{2 L \xi_{h,3}}{\mu^2} \left( \frac{1}{i+\theta_h} \right)^2 + \frac{4 L \xi_{h,3}}{\mu^2} \frac{1}{i+\theta_h+1} \cr
        &= \Delta_{i \wedge \tau} + \frac{2 L \xi_{h,3}}{\mu^2} \left( \frac{1}{(i+\theta_h)^2} + \frac{2}{i+\theta_h+1} \right) \cr
        &\overset{\mathrm{(i)}}{\leq} \Delta_{i \wedge \tau} + \frac{2 L \xi_{h,3}}{\mu^2} \left( \frac{2}{i+\theta_h} \right) \cr
        &= Y_i,
    \end{align}
    where (i) follows from $\theta_h \geq 2$ under parameter choice of Theorem~\ref{thm: policy_gradient}. This finishes the proof of $Y_i$ being a supermartingale. Now note that
    \begin{align*}
        \PP \{ \tau \leq T_h \} &= \PP \left\{\max_{i \in \{0,1,\dotsc,T_h\}} \Delta_{i} > 10 \zeta^{-1} \widetilde{J}_h (K_{h,0}) \right\} \cr
        &\leq \PP \left\{\max_{i \in \{0,1,\dotsc,T_h\}} \Delta_{i \wedge \tau} > 10 \zeta^{-1} \widetilde{J}_h (K_{h,0}) \right\} \cr
        &\overset{\mathrm{(i)}}{\leq} \PP \left\{\max_{i \in \{0,1,\dotsc,T_h\}} Y_i \geq 10 \zeta^{-1} \widetilde{J}_h (K_{h,0}) \right\} 
\end{align*}    
where (i) follows from the fact that $Y_i \geq \Delta_{i \wedge \tau}$. Using Doob/Ville's inequality for supermartingales, we have that 
\begin{align*}
        \PP \{ \tau \leq T_h \}&{\leq} \frac{\zeta \E[Y_0]}{10 \widetilde{J}_h (K_{h,0})} = \frac{\zeta \left(\Delta_0 + \frac{4 L \xi_{h,3}}{\mu^2} \frac{1}{\theta_h} \right)}{10 \widetilde{J}_h (K_{h,0})}.
\end{align*}
Using the choice of $\theta_h$ in Theorem~\ref{thm: policy_gradient}, we have that 
\begin{align*}
    \PP \{ \tau \leq T_h \}&{\leq} \frac{\zeta \left(\widetilde{J}_h (K_{h,0}) + 2 \widetilde{J}_h (K_{h,0}) \right)}{10 \widetilde{J}_h (K_{h,0})} = \frac{3}{10} \zeta.
\end{align*}
This verifies the second claim of Proposition~\ref{prop: lqr_policy}, concluding the proof.
\end{proof}

With Proposition~\ref{prop: lqr_policy} established, we are finally in a position to prove Theorem~\ref{thm: policy_gradient}:

\textit{Proof of Theorem~\ref{thm: policy_gradient}:}
We now employ Proposition~\ref{prop: lqr_policy} to validate the claims of Theorem~\ref{thm: policy_gradient}. Note that 
\begin{align*}
    \PP \left\{ \Delta_{T_h} \geq \frac{\mu}{4} \varsigma^2 \right\} &\leq \PP \left\{ \Delta_{T_h} 1_{\tau > T_h} \geq \frac{\mu}{4} \varsigma^2 \right\} + \PP \left\{ 1_{\tau \leq T_h} = 1 \right\} \\
    &\overset{\mathrm{(i)}}{\leq} \frac{4}{\mu \varsigma^2} \E[\Delta_{T_h} 1_{\tau > T_h}] + \PP \left\{ \tau \leq T_h \right\} \\
    &\overset{\mathrm{(ii)}}{\leq} \frac{7}{10} \zeta + \frac{3}{10} \zeta = \zeta,
\end{align*}
where (i) follows from applying Markov's inequality on the first claim of Proposition~\ref{prop: lqr_policy}, and (ii) comes directly from the second claim of Proposition~\ref{prop: lqr_policy}. Finally, we combine $\frac{\mu}{2}$-strong convexity of $\widetilde{J}_h$, along with $\nabla \widetilde{J}_h (\widetilde{K}^*_h) = 0$ to write
\begin{align*}
    \widetilde{J}_h (K_{h,T_h}) - \widetilde{J}_h (\widetilde{K}^*_h) &\geq \nabla \widetilde{J}_h (\widetilde{K}^*_h)^\top (K_{h,T_h} - \widetilde{K}^*_h) + \frac{\mu}{4} \| K_{h,T_h} - \widetilde{K}^*_h \|_F^2 \\
    &= \frac{\mu}{4} \| K_{h,T_h} - \widetilde{K}^*_h \|_F^2,
\end{align*}
and hence,
\begin{align*}
    \| K_{h,T_h} - \widetilde{K}^*_h \|_F^2 \leq \frac{4}{\mu} \left(\widetilde{J}_h (K_{h,T_h}) - \widetilde{J}_h (\widetilde{K}^*_h)\right) \leq \varsigma^2,
\end{align*}
with a probability of at least $1-\zeta$, finishing the proof. \proofend

\section{Sample complexity}
We now use our results for the inner and outer loops to obtain sample complexity bounds. In particular, Theorems~\ref{thm: policy_gradient} and~\ref{thm: LQR_DP_by_delta}, together with a union bound on the probabilities of failure at each step, yield the following result.
\begin{corollary}[End-to-end guarantee]\label{cor: overall_result}
    Choose
    \begin{align*}
        N = \frac{1}{2}\cdot \frac{\log\big(\frac{2\|Q_N-P^*\|_*\cdot\kappa_{P^*}\cdot \|A^*_K\|\cdot\|B\|} {\eps\cdot\lambda_{\min}(R)}\big)}{\log\big(\frac{1}{\|A^*_K\|_*}\big)} + 1,
    \end{align*}
    where $Q_N \succeq P^*$. Moreover, for each \(h \in \{0, 1, \dotsc, N-1\}\), let \(\varsigma_{h,\eps}\) be as defined in~\eqref{eq: def_varsigma_h}. Then Algorithm~\ref{alg:RHPG-RL} with the parameters as suggested in Theorem~\ref{thm: policy_gradient}, i.e., 
    \begin{align*}
        \alpha_{h,i} = \frac{2}{\mu} \frac{1}{i+\theta_h} \quad \text{for} \quad \theta_h = \max \{2,\frac{2 L \xi_{h,3}}{\mu^2 \widetilde{J}_h (K_{h,0})}\},
    \end{align*}
    and
    \[
    T_h = \frac{40}{7 \mu \varsigma_{h,\eps}^2 \zeta} \theta_h \widetilde{J}_h (K_{h,0}),
    \]
    outputs a control policy $\widetilde{K}_0$ that satisfies $\| \widetilde{K}_0 - K^* \| \leq \eps$ with a probability of at least $1 - N \zeta$. Additionally, for a sufficiently small $\eps$ such that $\eps < \frac{1 - \|A - BK^* \|_*}{\kappa_{P^*}^{1/2} \ \|B\|}$, $\widetilde{K}_0$ is stabilizing.
\end{corollary}

Corollary~\ref{cor: overall_result} gives the end-to-end guarantee that Algorithm~\ref{alg:RHPG-RL} computes a policy $\widetilde{K}_0$ satisfying $\|\widetilde{K}_0-K^*\|\leq\eps$ with high probability. Moreover, as mentioned in Subsection~\ref{sub: structure}, after the computation is complete, $\widetilde{K}_0$ is applied as a constant feedback gain to the original system, i.e., $u_t=-\widetilde{K}_0x_t$ for all $t\geq0$. Under the sufficient condition in Corollary~\ref{cor: overall_result}, this constant feedback policy is stabilizing. The following corollary gives the corresponding sample complexity bound.

\begin{corollary}\longthmtitle{Main result: complexity bound}
Algorithm~\ref{alg:RHPG-RL} achieves a sample complexity bound of at most
\[
\sum_{h=0}^{N-1} T_h = \widetilde{\mathcal{O}}\left(\eps^{-2}\right).
\]
\end{corollary}
It is worth comparing this result with the one in~\cite{XZ-TB:23}, 
taking into account the necessary adjustments ala Theorem~\ref{thm: LQR_DP}, where error accumulation results in a worse sample complexity bound. 
% Notably, if their approach were analyzed using either our framework or even their own, it would yield the following bound.

% \margin{We need to reword this. I tried. I think it is fine to drop our inner loop }
\begin{corollary}\longthmtitle{Prior Result: Complexity Bound}
Algorithm~1 in~\cite{XZ-TB:23} 
% If the approach in~\cite{XZ-TB:23} were analyzed using either our inner loop or their own, 
achieves the sample complexity bound of at most 
\[
\sum_{h=0}^{N-1} T'_h = \widetilde{\mathcal{O}}\left(\max \left\{\eps^{-2}, \eps^{- \left( 1 + \frac{\log(C_2)}{2\log{(1/\| A^*_K \|_*})} \right)} \right\}\right),
\]  
where \( T'_h \) denotes the counterpart of \( T_h \) in~\cite{XZ-TB:23}.
\end{corollary}

This comparison highlights the advantage of our method, which achieves a sample complexity bound of \( \widetilde{\mathcal{O}}(\eps^{-2}) \), with the exponent of $(1/\eps)$ independent of problem-specific constants. In contrast, the bound in~\cite{XZ-TB:23} becomes more restrictive as $C_2$ increases, since its second term scales as
\[
\widetilde{\mathcal{O}}\left(\eps^{-\left(1 + \frac{\log(C_2)}{2\log(1/\| A^*_K \|_*)} \right)}\right).
\]
This term can be \textit{arbitrarily larger} than $\widetilde{\mathcal{O}}(\eps^{-2})$, resulting in a much higher sample complexity in some cases.

% The results presented in Corollary~\ref{cor: overall_result} provide a rigorous theoretical foundation for Algorithm~\ref{alg:RHPG-RL}, establishing its ability to compute a control policy $\widetilde{K}_0$ satisfying $\| \widetilde{K}_0 - K^* \| \leq \eps$  with high probability, within a sample complexity bound of at most
% \begin{corollary}\longthmtitle{Main result: complexity bound}
% Algorithm~\ref{alg:RHPG-RL} satisfies the following bound:
% \[
% \sum_{h=0}^{N-1} T_h = \widetilde{\mathcal{O}}\left(\eps^{-\left(1 + \frac{\log(C_2)}{2\log{(1/\| A^*_K \|_*)}}\right)}\right).
% \] 

% \end{corollary}

% \begin{corollary}\longthmtitle{Prior result: complexity bound}
%     \[
% \sum_{h=0}^{N-1} T_h = \widetilde{\mathcal{O}}\left(\eps^{-\left(1 + \frac{\log(C_2)}{2\log{(1/\| A^*_K \|_*)}}\right)}\right).
% \] 

% \end{corollary}
Finally, to validate these theoretical guarantees and assess the algorithm’s empirical performance, we conduct simulation studies on a standard example from~\cite{XZ-TB:23}. The setup and results are presented in the following section.

\section{Simulation Studies}

\begin{figure}
\centering
\begin{subfigure}[b]{0.48\columnwidth}
    \includegraphics[width=\columnwidth]{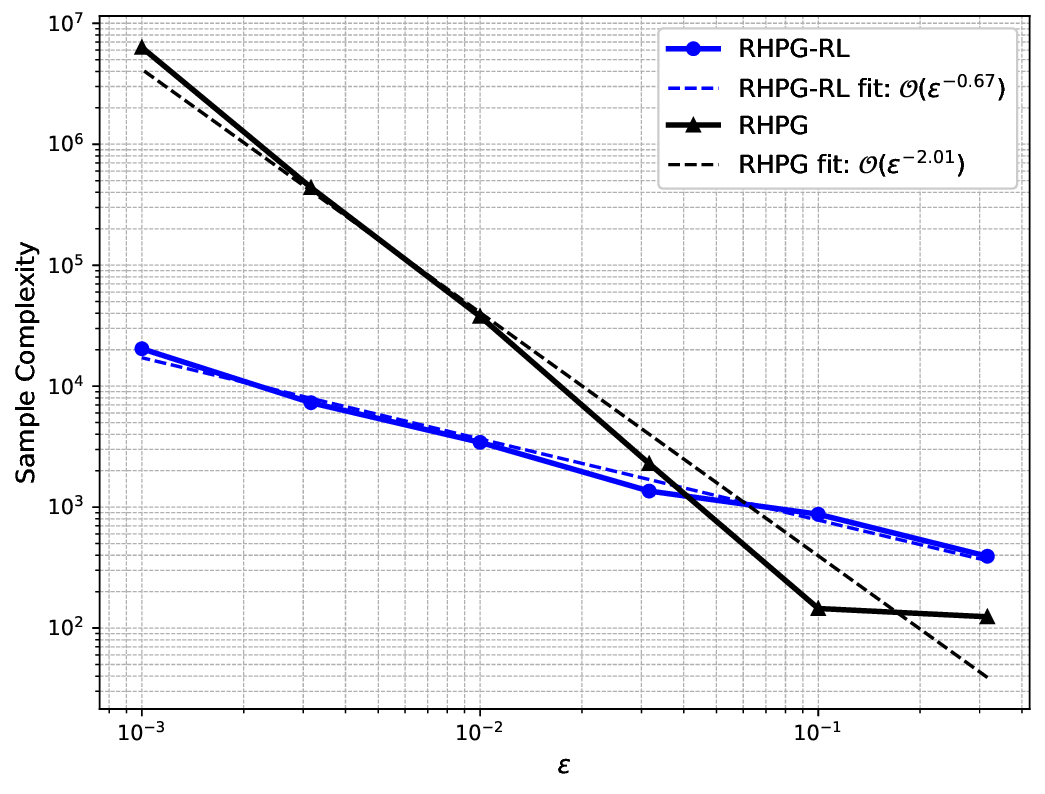}
    \vspace{-1.4em}\caption{Scalar example: comparison with RHPG.}
    \label{fig:sample_complexity_scalar}
\end{subfigure}
\hfill
\begin{subfigure}[b]{0.48\columnwidth}
    \includegraphics[width=\columnwidth]{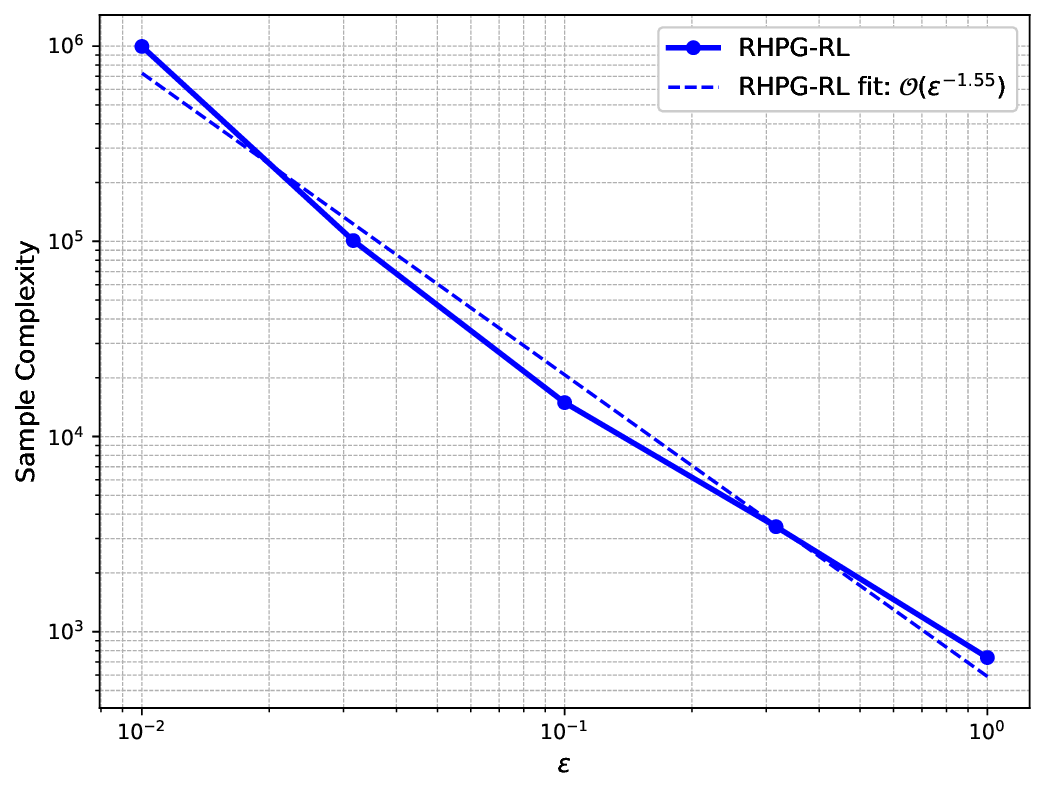}
    \vspace{-1.4em}\caption{Two-dimensional example.}
    \label{fig:sample_complexity_2d}
\end{subfigure}
\caption{Empirical sample complexity results averaged over 300 independent runs for each value of $\eps$. The dashed lines show the corresponding best power-law fits.}
\label{fig:simulation}
\end{figure}

We first consider the scalar example studied in~\cite{XZ-TB:23}, where $A=5$, $B=0.33$, $Q=R=1$, and the optimal policy is $K^*=14.5482$ with $P^*=221.4271$. We choose $Q_N=300\succeq P^*$ and initialize every inner loop with $K_{h,0}=0$. Since $A=5>1$, this initialization is clearly not stabilizing. We choose $N=\left\lceil \frac{1}{2}\log\left(\frac{1}{\eps}\right)\right\rceil+1$, consistent with~\eqref{eq: N_value_choice}, and consider
\[
    \eps\in\{10^{-0.5},10^{-1},10^{-1.5},10^{-2},10^{-2.5},10^{-3}\}.
\]

To avoid ambiguity, we refer to our proposed method in Algorithm~\ref{alg:RHPG-RL} as RHPG-RL and to the method in~\cite{XZ-TB:23} as RHPG. To provide a direct comparison, we implement RHPG on the same scalar example and show both methods in Fig.~\ref{fig:sample_complexity_scalar}. For RHPG, we use the same horizon $N=\left\lceil \frac{1}{2}\log\left(\frac{1}{\eps}\right)\right\rceil+1$ and terminal matrix $Q_N=300$, and otherwise follow the numerical choices in~\cite{XZ-TB:23}, including the smoothing radius $r=\sqrt{\eps}$ and a constant step-size of order $\mathcal{O}(\eps^2)$, consistent with~\cite[Proposition~1]{XZ-TB:23}. As shown in Fig.~\ref{fig:sample_complexity_scalar}, the empirical sample complexity of RHPG-RL scales approximately as $\mathcal{O}(\eps^{-0.67})$, while RHPG exhibits an empirical scaling of approximately $\mathcal{O}(\eps^{-2.01})$ on this example.

We next consider a two-dimensional example with
\[
    A=
    \begin{bmatrix}
    2 & 1\\
    1 & 1
    \end{bmatrix},
    \quad
    B=
    \begin{bmatrix}
    1 & 0\\
    1 & 1
    \end{bmatrix},
    \quad
    Q=
    \begin{bmatrix}
    4 & -2\\
    -2 & 4
    \end{bmatrix},
    \quad
    R=
    \begin{bmatrix}
    3 & -1\\
    -1 & 2
    \end{bmatrix}.
\]
Here, $\|P^*\|\approx15.38$, and we choose $Q_N=25I\succeq P^*$. As in the scalar example, every inner loop is initialized with $K_{h,0}=0$, and since $\rho(A)=(3+\sqrt{5})/2>1$, this initial policy is again not stabilizing. We use the same choice $N=\left\lceil \frac{1}{2}\log\left(\frac{1}{\eps}\right)\right\rceil+1$ and run the algorithm for
\[
    \eps\in\{10^0,10^{-0.5},10^{-1},10^{-1.5},10^{-2}\}.
\]
Figure~\ref{fig:sample_complexity_2d} shows an empirical sample-complexity scaling of approximately $\mathcal{O}(\eps^{-1.55})$. This is visibly steeper than in the scalar example, yet it remains within our worst-case theoretical guarantee of $\widetilde{\mathcal{O}}(\eps^{-2})$.

\section{Conclusion}
In this paper, we introduced a novel approach to solving the model-free LQR problem, inspired by policy gradient methods, particularly REINFORCE. Our algorithm eliminates the restrictive requirement of starting with a stable initial policy, making it applicable in scenarios where obtaining such a policy is challenging. Furthermore, it removes the reliance on two-point gradient estimation, enhancing practical applicability while maintaining similar rates.

Beyond these improvements, we introduced a refined outer-loop analysis that significantly improves error propagation, leveraging the contraction of the Riccati operator under the Riemannian distance. This ensures that the accumulated error remains linear in the horizon length, leading to a sample complexity bound of $\widetilde{\mathcal{O}}(\eps^{-2})$, independent of problem-specific constants, making the method more broadly applicable.

We provide a rigorous theoretical analysis establishing convergence to the optimal policy with competitive sample complexity. Our simulations also show favorable empirical performance without requiring two-point gradient estimates. Together, these results highlight the potential of the proposed method for model-free control. Future work includes extensions to nonlinear and partially observed systems, as well as robustness improvements.

% \bibliographystyle{amsplain}
% \bibliography{alias,Group-bib}

\appendix

\section{Proof of Theorem~\ref{thm: RDE converges to ARE}} \label{app: RDE converges to ARE}
% We first identify one-to-one correspondences between system parameters in LQR and those in Kalman filtering \cite{zhang2023learning}:
% \begin{center}
% 	\begin{tabular}{|l|l|l|l|l|l|}
% \cline{1-6}
% LQR: & $A$ & $B$ & $Q$ & $R$ & $Q_N$ \\ 
%      & \hspace{0.2em}$\updownarrow$ & \hspace{0.2em}$\updownarrow$ & \hspace{0.2em}$\updownarrow$ & \hspace{0.2em}$\updownarrow$ & \hspace{0.2em}$\updownarrow$\\ 
% KF \cite{zhang2023learning}:  & $A^{\top}$ & $C^{\top}$ & $W$ & $V$ & $X_0$\\ \cline{1-6}             
% \end{tabular}
% \end{center}
% We also identify direct correspondences between our $P_t$, $P^*$, $K_t$, and $K^*$ and \cite{zhang2023learning}'s $\Sigma_{N-t}$, $\Sigma^*$, $L_{N-1-t}$, and $L^*$, respectively. Then, letting 
We let
\begin{align*}
    &\overline{P}_t := P^*_t- P^*, \quad  \overline{R} := R + B^{\top}P^*B, \quad \overline{A} := A - B\overline{R}^{-1}B^{\top}P^*A,
\end{align*}
% Since $P_N^*=Q_N\succeq P^*$ and $\mathcal{R} (\cdot)$ is order-preserving, induction gives $P_t^*\succeq P^*$ for all $t\in\{0,1,\dotsc,N\}$, and hence $\overline{P}_t\succeq 0$. We then have
and we have
\begin{align}
    \overline{P}_{t} &= \overline{A}^{\top}\overline{P}_{t+1}\overline{A} - \overline{A}^{\top}\overline{P}_{t+1}B(\overline{R} + B^{\top}\overline{P}_{t+1}B)^{-1}B^{\top}\overline{P}_{t+1}\overline{A} \nonumber\\
    &= \overline{A}^{\top}\overline{P}_{t+1}^{1/2}\big[I + \overline{P}^{1/2}_{t+1}B\overline{R}^{-1}B^{\top}\overline{P}_{t+1}^{1/2}\big]^{-1}\overline{P}_{t+1}^{1/2}\overline{A} \nonumber\\
    &\leq \big[1+\lambda_{\min}(\overline{P}^{1/2}_{t+1}B\overline{R}^{-1}B^{\top}\overline{P}_{t+1}^{1/2})\big]^{-1}\overline{A}^{\top}\overline{P}_{t+1}\overline{A} \nonumber\\
    &=: \mu_{t} \overline{A}^{\top}\overline{P}_{t+1}\overline{A},\label{eqn:RDE_conv_step2}
\end{align}
where $\overline{P}_{t+1}^{1/2}$ denotes the unique positive semi-definite (psd) square root of the psd matrix $\overline{P}_{t+1}$, $0 < \mu_t \leq 1$ for all $t$, and $\overline{A}$ satisfies $\rho(\overline{A}) < 1$. We now use $\|\cdot\|_*$ to represent the $P^*$-induced matrix norm and invoke \cite[Theorem 14.4.1]{BH-AHS-TK:99}, where our $\overline{P}_t$, $\overline{A}^{\top}$ and $P^*$ correspond to $P_i - P^*$, $F_p$ and $W$ in \cite{BH-AHS-TK:99}, respectively. By \cite[Theorem 14.4.1]{BH-AHS-TK:99} and \eqref{eqn:RDE_conv_step2}, we obtain $\|\overline{A}\|_* < 1$ and given that $\mu_t \leq 1$, 
\begin{align*}
    \|\overline{P}_{t}\|_* \leq \|\overline{A}\|^2_*   \|\overline{P}_{t+1}\|_*.
\end{align*}
Therefore, the convergence is exponential such that $\|\overline{P}_t\|_* \leq \|\overline{A}\|_*^{2(N-t)}  \|\overline{P}_{N}\|_*$. Consequently, the convergence of $\overline{P}_t$ to $0$ in spectral norm can be characterized as
\begin{align}
    \|\overline{P}_t\| \leq \kappa_{P^*}  \|\overline{P}_t\|_* \leq \kappa_{P^*} \|\overline{A}\|_*^{2(N-t)}  \|\overline{P}_{N}\|_*, \label{eq:exponential_conv_of_P_star_t_diff}
\end{align}
and hence
\begin{align}
    \|P_t^*\| \leq \|P^*\|+\kappa_{P^*}\|Q_N-P^*\|_* = P_{\max},\label{eq: P_star_uniform_bound}
\end{align}
for all $t\in\{0,1,\dotsc,N\}$. Also as a result of~\eqref{eq:exponential_conv_of_P_star_t_diff}, to ensure $\|\overline{P}_1\| \leq \eps$, it suffices to require
\begin{align}\label{eqn:required_time}
    N \geq \frac{1}{2}\cdot \frac{\log\big(\frac{\|\overline{P}_N\|_*\cdot \kappa_{P^*}}{\eps}\big)}{\log\big(\frac{1}{\|\overline{A}\|_*}\big)} + 1.
\end{align}
Lastly, we show that the convergence of $K^*_t$ to $K^*$ follows from the convergence of $P^*_t$ to $P^*$. This can be verified through:
\begin{align}
    K^*_t - K^* &= (R+B^{\top}P^*_{t+1}B)^{-1}B^{\top}P^*_{t+1}A - (R+B^{\top}P^*B)^{-1}B^{\top}P^*A \nonumber\\
    &=\big[(R+B^{\top}P^*_{t+1}B)^{-1}-(R+B^{\top}P^*B)^{-1}\big]B^{\top}P^*A + (R+B^{\top}P^*_{t+1}B)^{-1}B^{\top}(P^*_{t+1}-P^*)A \nonumber\\
    &= (R+B^{\top}P^*_{t+1}B)^{-1}B^{\top}(P^*-P^*_{t+1})BK^* -(R+B^{\top}P^*_{t+1}B)^{-1}B^{\top}(P^*-P^*_{t+1})A\nonumber\\
    &=(R+B^{\top}P^*_{t+1}B)^{-1}B^{\top}(P^*-P^*_{t+1})(BK^*-A). \label{eqn:kdiff}
\end{align}
Hence, we have $\|K^*_t - K^*\| \leq \frac{\|\overline{A}\|\cdot \|B\|}{\lambda_{\min}(R)}\cdot \|P^*_{t+1} - P^*\|$ and
\begin{align*}
    \|K^*_0 - K^*\| \leq \frac{\|\overline{A}\|\cdot \|B\|}{\lambda_{\min}(R)}\cdot \|\overline{P}_1\|.
\end{align*}
Substituting $\eps$ in \eqref{eqn:required_time} with $\frac{\eps\cdot\lambda_{\min}(R)}{\|\overline{A}\|\cdot\|B\|}$ completes the proof.
\begin{remark} \label{rem: Q_N inherent_assumption}
    Since Theorem~\ref{thm: RDE converges to ARE} requires $\overline{P}_t = P^*_t - P^*$ to be positive semidefinite for all $t$, it implies that we have access to a $P^*_N = Q_N$ that satisfies $Q_N \succeq P^*$. Conversely, this condition is sufficient to guarantee the required positive semidefiniteness. Indeed, since the Riccati operator is order-preserving, i.e., $U \succeq V$ implies $\mathcal{R}(U) \succeq \mathcal{R}(V)$, and $\mathcal{R}(P^*)=P^*$, it follows inductively from $P^*_N=Q_N\succeq P^*$ that
    \[
        P^*_t \succeq P^*,
    \]
    for all $t\in\{0,1,\dotsc,N\}$, satisfying said requirement. \oprocend
\end{remark}

\section{Proof of Theorem~\ref{thm: LQR_DP_by_delta}} \label{app: thm_LQR_DP_by_Delta}
We start the proof by providing some preliminary results.
\begin{lemma} \label{lem: delta_upper_lower}
    Let $U$ and $V$ be two positive definie matrices. It holds that
    \begin{equation}
    \|U-V\| \le \|V\|\, e^{\delta(U,V)}\, \delta(U,V). \label{eq: delta_lower}
    \end{equation}
    Furthermore, if 
    \begin{equation}
        \|V^{-1}\|\,\|U-V\| < 1, \label{eq: delta_upper_assumption}
    \end{equation}
    then we have
    \begin{equation}
   \delta(U,V) \le \frac{\|V^{-1}\|\,\|U-V\|_F}{1-\|V^{-1}\|\,\|U-V\|}. \label{eq: delta_upper}
   \end{equation}
\end{lemma}
\begin{proof}
    First, since \(U\) and \(V\) are positive definite, we have that \(V^{-1/2} U V^{-1/2}\)  is positive definite, and therefore has a logarithm; so we let
    \[
    Z := \log(V^{-1/2} U V^{-1/2}),
    \]
    and hence, we can write
    \[
    U = V^{1/2} \exp(Z) V^{1/2}.
    \]
    The eigenvalues of \(Z\) are precisely the logarithms of the eigenvalues of \(UV^{-1}\) due to $UV^{-1}$ and $V^{-1/2} U V^{-1/2}$ being similar. As a result,
    \[
    \delta(U,V) = \|Z\|_F.
    \]
    We now write
    \[
    U - V = V^{1/2}\exp(Z)V^{1/2} - V = V^{1/2}(\exp(Z)-I)V^{1/2},
    \]
    and thus,
    \begin{equation}
    \|U-V\| \le \|V\|\, \|\exp(Z)-I\|. \label{eq: delta_lower_prelim1}
    \end{equation}
    Since $e^{x}-1\leq  x e^{x}$ whenever $x \geq 0$, we also have for any matrix $Z$, by consider the expansion of $e^Z$: 
    \[
    \|\exp(Z)-I\| \leq e^{||Z||}  - 1\le e^{\|Z\|}\, \|Z\|.
    \]
    Since the spectral norm is always bounded by the Frobenius norm, we have:
    \[
    \|\exp(Z)-I\| \le e^{\|Z\|_F}\, \|Z\|_F.
    \]
    Finally, recalling that \(\|Z\|_F = \delta(U,V)\), this becomes:
    \[
    \|\exp(Z)-I\| \le e^{\delta(U,V)}\, \delta(U,V),
    \]
    which after substituting into~\eqref{eq: delta_lower_prelim1} yields:
    \[
    \|U-V\| \le \|V\|\, e^{\delta(U,V)}\, \delta(U,V),
    \]
    concluding the proof of the first claim. We now move on to the second claim. As before, we write 
    \[
    \delta(U,V)=\|\log(V^{-1/2}UV^{-1/2})\|_F.
    \]
    We now define
    \[
    X := V^{-1/2}(U-V)V^{-1/2},
    \]
    so that
    \[
    V^{-1/2}UV^{-1/2} = I + X.
    \]
    Moreover, following~\eqref{eq: delta_upper_assumption},
    \[
    \| X \| = \|V^{-1/2}(U-V)V^{-1/2}\| \leq \| V^{-1} \| \| U - V \| < 1,
    \]
    and hence, one can use the series expansion of the logarithm
    \[
    \log(I+X)=X - \frac{1}{2}X^2 + \cdots,
    \]
    to show
    \begin{align}
        \|\log(I+X)\|_F &= \left\| \sum_{k=1}^{\infty} \frac{(-1)^{k+1}}{k} X^k \right\|_F \cr
        &\leq \sum_{k=1}^{\infty} \| X^k \|_F \cr
        &\leq \sum_{k=1}^{\infty} \| X \|_F \| X^{k-1} \| \cr
        &\leq \| X \|_F \sum_{k=0}^{\infty} \| X \|^{k} \cr
        &= \frac{\|X\|_F}{1-\|X\|}.
    \end{align}
    As a result, we have
    \begin{align*}
    \delta(U,V) &= \|\log(I+X)\|_F \\
    &\le \frac{\|X\|_F}{1-\|X\|} \\
    &= \frac{\|V^{-1/2}(U-V)V^{-1/2}\|_F}{1-\|V^{-1/2}(U-V)V^{-1/2}\|} \\
    &\leq \frac{\|V^{-1}\| \| (U-V) \|_F}{1-\|V^{-1}\| \| (U-V) \|},
    \end{align*}
    finishing the proof.
\end{proof}

Building on Lemma~\ref{lem: delta_upper_lower}, we proceed to state the following result regarding the LQR setting.
\begin{lemma} \label{lem: delta_uniform_upper_induction}
    Let $t \in \{1,2,\dotsc,N-1\}$ and select $Q_N \succeq Q$. Additionally, assume that for all \( t' \in \{ t+1, t+2, \dotsc, N\} \), we have 
    \begin{align}
        &\| P^*_{t'} - \widetilde{P}_{t'} \| \leq a, \quad \text{and} \label{eq: a_condition_t'} \\
        &\delta(\widetilde{P}^*_{t'}, \widetilde{P}_{t'} ) \leq \varepsilon, \label{eq: delta_upper_condition_t'}
    \end{align}
    where \( \varepsilon \) satisfies
    \begin{equation}
    \varepsilon \leq \frac{1}{N} \min \left\{ \frac{a}{2 e P_{\max}}, 1 \right\}. \label{eq: varepsilon_condition}
    \end{equation}
    If
    \begin{equation}
    \| \widetilde{K}_t - \widetilde{K}^*_t \|_F \leq \sqrt{\frac{a}{C_3} \varepsilon}, \label{eq: K_diff_condition_induction}
    \end{equation}
    then the following bounds hold:
    \begin{align}
        &\| P^*_t - \widetilde{P}_t \| \leq a, \quad \text{and} \label{eq: a_condition} \\
        &\delta(\widetilde{P}^*_t, \widetilde{P}_t ) \leq \varepsilon. \label{eq: delta_upper_condition}
    \end{align}
\end{lemma}
\begin{proof}
    Before we move on to the proof, we esablish some preliminary results.
    First, note that since $P_N^*=Q_N\succeq Q\succ0$, it follows directly from~\eqref{eq: RDE} by backward induction that $P_t^*\succeq Q$ for all $t\in\{1,2,\dotsc,N\}$. Therefore,
    \begin{equation}
        \sigma_{\min}(P_t^*)\geq \sigma_{\min}(Q)=2a>0.
        \label{eq: sigma_min_P_star_t_lowerbound}
    \end{equation}
    Moreover, due to~\eqref{eq: a_condition_t'}, we have
    \begin{equation}
        \widetilde{P}_{t'} \succeq P^*_{t'} - aI \succeq a I \succ 0 \label{eq: sigma_min_P_tilde_star_t'_lowerbound}
    \end{equation}
    for all \( t' \in \{ t+1, t+2, \dotsc, N\} \). Since~\eqref{eq: sigma_min_P_star_t_lowerbound} and~\eqref{eq: sigma_min_P_tilde_star_t'_lowerbound} hold, we can apply Lemma~\ref{lem: delta_Riccati_contraction} to show that for all \( t' \in \{ t+1, t+2, \dotsc, N\} \),
    \begin{align}
    \delta (P_{t'-1}^*, \widetilde{P}_{t'-1}^*) &\overset{\mathrm{(i)}}{=}
    \delta(\mathcal{R}(P^*_{t'}),\mathcal{R}(\widetilde{P}_{t'})) \cr
    &\leq \delta(P^*_{t'}, \widetilde{P}_{t'}), \label{eq: delta_contraction_t'}
    \end{align}
    where (i) follows from~\eqref{eq: Riccati_operator_on_P_tilde_star} and~\eqref{eq: Riccati_operator_on_P_star}. Following~\eqref{eq: delta_contraction_t'}, we can now write
    
    % \begin{align}
    %     \sigma_{\min} (\widetilde{P}_{t'}) &\geq \sigma_{\min} (P^*_{t'}) - a \cr
    %     &\geq \sigma_{\min} (Q) - a \cr
    %     &= a > 0.
    % \end{align}
    \begin{align}
        \delta(P^*_t, \widetilde{P}^*_t) &\leq \delta(P^*_{t+1}, \widetilde{P}_{t+1}) \cr
        &\overset{\mathrm{(i)}}{\leq} \delta(P^*_{t+1}, \widetilde{P}^*_{t+1}) + \delta(\widetilde{P}^*_{t+1}, \widetilde{P}_{t+1}) \cr
        &\leq \delta(P^*_{t+2}, \widetilde{P}_{t+2}) + \delta(\widetilde{P}^*_{t+1}, \widetilde{P}_{t+1}) \cr
        &\leq \delta(P^*_{t+2}, \widetilde{P}^*_{t+2}) + \delta(\widetilde{P}^*_{t+2}, \widetilde{P}_{t+2}) + \delta(\widetilde{P}^*_{t+1}, \widetilde{P}_{t+1}) \cr
        &\leq \cdots \cr
        &\leq \delta(P^*_{N}, \widetilde{P}_{N}) + \sum_{k=1}^{N-t-1} \delta(\widetilde{P}^*_{t+k}, \widetilde{P}_{t+k}) \cr
        &\overset{\mathrm{(ii)}}{=} \sum_{k=1}^{N-t-1} \delta(\widetilde{P}^*_{t+k}, \widetilde{P}_{t+k}) \cr
        &\overset{\mathrm{(iii)}}{\leq} \varepsilon N, \label{eq: geometric_sum}
    \end{align}
    where (i) is due to the triangle inequality of the Riemannian distance~\cite{RB:07}, (ii) follows from $P^*_{N} = \widetilde{P}_{N} = Q_N$, and (iii) from~\eqref{eq: delta_upper_condition_t'}.
    We now start the proof of~\eqref{eq: a_condition} by writing
    \begin{align}
        \| P^*_t - \widetilde{P}_t \| \leq \| P^*_t - \widetilde{P}^*_t \| + \| \widetilde{P}^*_t - \widetilde{P}_t \|, \label{eq: P_diff_total}
    \end{align}
    and trying to provide a bound for both terms of the right-hand side of~\eqref{eq: P_diff_total}. For the first term, we have
    \begin{align}
        \| P^*_t - \widetilde{P}^*_t \| &\overset{\mathrm{(i)}}{\leq} P_{\max} e^{\delta (P^*_t, \widetilde{P}^*_t)} \delta (P^*_t, \widetilde{P}^*_t) \cr
        &\overset{\mathrm{(ii)}}{\leq} P_{\max} e^{\varepsilon N} \varepsilon N \cr
        &\overset{\mathrm{(iii)}}{\leq} \frac{a}{2}, \label{eq: P_diff_1}
    \end{align}
    where (i) follows from~\eqref{eq: delta_lower}, (ii) from~\eqref{eq: geometric_sum}, and (iii) from the condition on $\varepsilon$ in~\eqref{eq: varepsilon_condition}. As for the second term on the right-hand side of~\eqref{eq: P_diff_total}, we can write
    \begin{align}
        \widetilde{P}^*_{t} - \widetilde{P}_{t} &= (A-B\widetilde{K}^*_t)^{\top}\widetilde{P}_{t+1}(A-B\widetilde{K}^*_t) + (\widetilde{K}^*_t)^{\top}R\widetilde{K}^*_t - (A-B\widetilde{K}_t)^{\top}\widetilde{P}_{t+1}(A-B\widetilde{K}_t) - (\widetilde{K}_t)^{\top}R\widetilde{K}_t \nonumber\\
        &= - (\widetilde{K}^*_t)^{\top}B^{\top}\widetilde{P}_{t+1}A - A^{\top}\widetilde{P}_{t+1}B\widetilde{K}_t^* + (\widetilde{K}^*_t)^{\top}(R+B^{\top}\widetilde{P}_{t+1}B)\widetilde{K}^*_t   \nonumber\\
        &\quad + \widetilde{K}_t^{\top}B^{\top}\widetilde{P}_{t+1}A + A^{\top}\widetilde{P}_{t+1}B\widetilde{K}_t - \widetilde{K}_t^{\top}(R+B^{\top}\widetilde{P}_{t+1}B)\widetilde{K}_t \nonumber\\
        &\overset{\mathrm{(i)}}{=} \big[(R+B^{\top}\widetilde{P}_{t+1}B)^{-1}B^{\top}\widetilde{P}_{t+1}A - \widetilde{K}^*_t\big]^{\top} (R+B^{\top}\widetilde{P}_{t+1}B) \big[(R+B^{\top}\widetilde{P}_{t+1}B)^{-1}B^{\top}\widetilde{P}_{t+1}A - \widetilde{K}^*_t \big] \nonumber\\
        &\quad - \big[(R+B^{\top}\widetilde{P}_{t+1}B)^{-1}B^{\top}\widetilde{P}_{t+1}A - \widetilde{K}_t\big]^{\top} (R+B^{\top}\widetilde{P}_{t+1}B) \big[(R+B^{\top}\widetilde{P}_{t+1}B)^{-1}B^{\top}\widetilde{P}_{t+1}A - \widetilde{K}_t \big] \nonumber \\
        &= \big[\widetilde{K}^*_t - \widetilde{K}^*_t\big]^{\top}(R+B^{\top}\widetilde{P}_{t+1}B) \big[\widetilde{K}^*_t - \widetilde{K}^*_t \big] - \big[\widetilde{K}^*_t - \widetilde{K}_t\big]^{\top}(R+B^{\top}\widetilde{P}_{t+1}B) \big[\widetilde{K}^*_t - \widetilde{K}_t \big] \nonumber \\
        &= - \big[\widetilde{K}^*_t - \widetilde{K}_t\big]^{\top}(R+B^{\top}\widetilde{P}_{t+1}B) \big[\widetilde{K}^*_t - \widetilde{K}_t \big], \label{eq: P_diff_2_prelim}
    \end{align}
    where (i) follows from completion of squares. Combining~\eqref{eq: P_diff_2_prelim} and~\eqref{eq: a_condition_t'}, we have
    \begin{align}
        \| \widetilde{P}^*_{t} - \widetilde{P}_{t} \| &\leq \| R+(P_{\max}+a) B^{\top}B \| \| \widetilde{K}^*_t - \widetilde{K}_t \|^2 \label{eq: delta_upper_numerator_prelim} \\
        &\leq \frac{C_3}{2} \| \widetilde{K}^*_t - \widetilde{K}_t \|^2 \cr
        &\leq \left( \frac{C_3}{2} \right) \left( \frac{a}{C_3} \varepsilon \right) \cr
        &\leq \frac{a}{2}. \label{eq: P_diff_2}
    \end{align}
    Finally, substituting~\eqref{eq: P_diff_1} and~\eqref{eq: P_diff_2} in~\eqref{eq: P_diff_total}, we have
    \[
    \| P^*_t - \widetilde{P}_t \| \leq \frac{a}{2} + \frac{a}{2} = a,
    \]
    finishing the proof of~\eqref{eq: a_condition}. Having established this, we proceed to prove~\eqref{eq: delta_upper_condition}. Note that similar to~\eqref{eq: delta_upper_numerator_prelim}, we can write
    \begin{align}
        \| \widetilde{P}^*_{t} - \widetilde{P}_{t} \|_F &\leq \| R+(P_{\max}+a) B^{\top}B \| \| \widetilde{K}^*_t - \widetilde{K}_t \|_F^2 \cr
        &\leq \frac{C_3}{2} \| \widetilde{K}^*_t - \widetilde{K}_t \|_F^2. \label{eq: delta_upper_numerator}
    \end{align}
    Moreover, due to~\eqref{eq: a_condition}, we have that $\widetilde{P}_t \succeq P^*_t - aI$, and hence,
    \begin{align}
        \sigma_{\min} (\widetilde{P}_t) \geq \sigma_{\min} (P^*_t) - a \overset{\mathrm{(i)}}{\geq} a, \label{eq: P_tilde_t_sigma_min_lower}
    \end{align}
    where (i) follows from~\eqref{eq: sigma_min_P_star_t_lowerbound}. Combining~\eqref{eq: P_tilde_t_sigma_min_lower} and~\eqref{eq: P_diff_2} yields
    \begin{align}
    \| \widetilde{P}_t^{-1} \| \| \widetilde{P}^*_{t} - \widetilde{P}_{t} \| &= \frac{\| \widetilde{P}^*_{t} - \widetilde{P}_{t} \|}{\sigma_{\min} (\widetilde{P}_t)} \leq \frac{a/2}{a} = \frac{1}{2}. \label{eq: delta_upper_condition_met_one_half}
    \end{align}
    Thus, the condition~\eqref{eq: delta_upper_assumption} of Lemma~\ref{lem: delta_upper_lower} is met, and we can employ~\eqref{eq: delta_upper} to write
    \begin{align*}
        \delta(\widetilde{P}^*_{t}, \widetilde{P}_{t}) &\leq \frac{\| \widetilde{P}_t^{-1} \| \| \widetilde{P}^*_{t} - \widetilde{P}_{t} \|_F}{1 - \| \widetilde{P}_t^{-1} \| \| \widetilde{P}^*_{t} - \widetilde{P}_{t} \|} \\
        &\overset{\mathrm{(i)}}{\leq} \frac{(1/a) \| \widetilde{P}^*_{t} - \widetilde{P}_{t} \|_F}{(1/2)} \\
        &\overset{\mathrm{(ii)}}{\leq} \frac{C_3}{a} \| \widetilde{K}^*_t - \widetilde{K}_t \|_F^2 \\
        &\overset{\mathrm{(iii)}}{\leq} \varepsilon,
    \end{align*}
    where (i) follows from~\eqref{eq: P_tilde_t_sigma_min_lower} and~\eqref{eq: delta_upper_condition_met_one_half}, (ii) from~\eqref{eq: delta_upper_numerator}, and (iii) from~\eqref{eq: K_diff_condition_induction}. This verifies~\eqref{eq: delta_upper_condition}, concluding the proof.
\end{proof}
Having established Lemma~\ref{lem: delta_uniform_upper_induction}, we can finally present the proof of~\ref{thm: LQR_DP_by_delta}.

\textit{Proof of Theorem~\ref{thm: LQR_DP_by_delta}:}
First, according to Theorem~\ref{thm: RDE converges to ARE}, our choice of $N$ in~\eqref{eq: N_value_choice} ensures that $K^*_{0}$ is stabilizing and $\|K^*_{0} - K^*\| \leq \eps/2$. Then, it remains to show that the output $\widetilde{K}_0$ satisfies $\|\widetilde{K}_0 - K^*_0\| \leq \eps/2$. 

Now observe that
\begin{align*}
\|\widetilde{K}_{0} - K^*_{0}\| \leq \|\widetilde{K}^*_{0} - K^*_{0}\| + \|\widetilde{K}_0-\widetilde{K}_0^*\|,
\end{align*}
where substituting $K^*_t$ and $K^*$ in \eqref{eqn:kdiff}, respectively, with $\widetilde{K}^*_0$ and $K^*_0$ leads to
\begin{align*}
\widetilde{K}^*_{0} - K^*_{0} = (R+B^{\top}\widetilde{P}_{1}B)^{-1}B^{\top}(P^*_1-\widetilde{P}_{1})(BK^*_0-A).
\end{align*}
Hence, the error size $\|\widetilde{K}^*_{0} - K^*_{0}\|$ could be bounded by 
\begin{align}
\|\widetilde{K}^*_{0} - K^*_{0}\| \leq \frac{\|A-BK^*_0\|\cdot\|B\|}{\lambda_{\min}(R)}\cdot\|P^*_1-\widetilde{P}_{1}\| \label{eq: laststep_req}.
\end{align}
% Define the helper constants
% \begin{align*}
% 	C_1:= \frac{\varphi\cdot\|B\|}{\lambda_{\min}(R)} > 0, \ \varphi := \max_{t\in\{0, \cdots, N-1\}}\|A-BK^*_t\|.
% \end{align*}
Next, since we have $\|\widetilde{K}_0-\widetilde{K}_0^*\| \leq \varsigma_{0,\eps} = \eps/4$, it suffices to show $\|\widetilde{K}^*_{0} - K^*_{0}\| \leq \eps/4$ to fulfill $\|\widetilde{K}_{0} - K^*_{0}\| \leq \eps/2$. 
% We 
% % select a fixed scalar $a > 0$ that is independent of system parameters and $\eps$, and 
% additionally require $\|P^*_1-\widetilde{P}_{1}\| \leq a$ to upper-bound the positive definite solutions of \eqref{eq: lqr_lyapunov}. 
Then, by \eqref{eq: laststep_req}, in order to satisfy $\|\widetilde{K}^*_{0} - K^*_{0}\| \leq \eps/4$, it remains to show
\begin{align}\label{eq: P1_req}
\|P^*_1-\widetilde{P}_{1}\| \leq \frac{\eps}{4 C_1}.
\end{align}
% Subsequently, we have
% \begin{align}\label{eqn:plast}
%     P^*_1-\widetilde{P}_{1} = (P^*_1 - \widetilde{P}^*_{1}) + (\widetilde{P}^*_{1} - \widetilde{P}_{1}).
% \end{align}
In order to show this, we first let
\begin{equation}
\varepsilon = \frac{1}{N} \min \left\{ \frac{\eps}{8 e C_1 P_{\max}}, \frac{a}{2 e P_{\max}}, 1 \right\}, \label{eq: varepsilon_value}
\end{equation}
which clearly satisfies~\eqref{eq: varepsilon_condition}. Now we want to show, by strong induction, that
\begin{align*}
    &\| P^*_t - \widetilde{P}_t \| \leq a, \quad \text{and} \\
    &\delta(\widetilde{P}^*_t, \widetilde{P}_t) \leq \varepsilon,
\end{align*}
for all $t \in \{N,N-1,\dotsc,1\}$. For the base case, we have
\[
P^*_N = \widetilde{P}^*_N = \widetilde{P}_N = Q_N,
\]
and hence, it immediately follows that
\begin{align*}
    &\| P^*_N - \widetilde{P}_N \| = 0 \leq a, \quad \text{and} \\
    &\delta(\widetilde{P}^*_N, \widetilde{P}_N) = 0 \leq \varepsilon.
\end{align*}
Now since it holds in the statement of Theorem~\ref{thm: LQR_DP_by_delta} that
\[
\| \widetilde{K}_h - \widetilde{K}^*_h \| \leq \varsigma_{h,\eps} \leq \sqrt{\frac{a}{C_3} \varepsilon},
\]
which satisfies~\eqref{eq: K_diff_condition_induction}, the inductive step follows directly from Lemma~\ref{lem: delta_uniform_upper_induction}. We have now succesfully established that
\begin{align}
    &\| P^*_t - \widetilde{P}_t \| \leq a, \quad \text{and} \label{eq: p_diff_leq_a}\\
    &\delta(\widetilde{P}^*_t, \widetilde{P}_t) \leq \varepsilon, \label{eq: delta_leq_varepsilon}
\end{align}
for all $t \in \{N,N-1,\dotsc,1\}$.
As a result, we have
\begin{align}
    \delta(P_1^*, \widetilde{P}_1^*) &\leq \delta(P_{2}^*, \widetilde{P}_{2}) \cr
    &\leq \delta(P_{2}^*, \widetilde{P}^*_{2}) + \delta(\widetilde{P}_{2}^*, \widetilde{P}_{2}) \cr
    &\leq \delta(P_{3}^*, \widetilde{P}_{3}) + \delta(\widetilde{P}_{2}^*, \widetilde{P}_{2}) \cr
    &\leq \delta(P_{3}^*, \widetilde{P}^*_{3}) + \delta(\widetilde{P}_{3}^*, \widetilde{P}_{3}) + \delta(\widetilde{P}_{2}^*, \widetilde{P}_{2}) \cr
    &\leq \cdots \cr
    &\leq \delta(P_{N}^*, \widetilde{P}_{N}) + \sum_{k=2}^{N-1} \delta(\widetilde{P}_{k}^*, \widetilde{P}_{k}) \cr
    &= \sum_{k=2}^{N-1} \delta(\widetilde{P}_{k}^*, \widetilde{P}_{k}) \cr
    &\leq \varepsilon N. \label{eq: geometric_sum_2}
\end{align}
We now show~\eqref{eq: P1_req} by writing
\begin{align}
    \| P^*_1 - \widetilde{P}_1 \| \leq \| P^*_1 - \widetilde{P}^*_1 \| + \| \widetilde{P}^*_1 - \widetilde{P}_1 \|, \label{eq: P1_diff_total}
\end{align}
and providing a bound for both terms of the right-hand side of~\eqref{eq: P1_diff_total}. For the first term, we have
\begin{align}
    \| P^*_1 - \widetilde{P}^*_1 \| &\overset{\mathrm{(i)}}{\leq} P_{\max} e^{\delta (P^*_1, \widetilde{P}^*_1)} \delta (P^*_1, \widetilde{P}^*_1) \cr
    &\overset{\mathrm{(ii)}}{\leq} P_{\max} e^{\varepsilon N} \varepsilon N \cr
    &\overset{\mathrm{(iii)}}{\leq} \frac{\eps}{8 C_1}, \label{eq: P1_diff_1}
\end{align}
where (i) follows from Lemma~\ref{lem: delta_upper_lower}, (ii) from~\eqref{eq: geometric_sum_2}, and (iii) from~\eqref{eq: varepsilon_value}. As for the second term on the right-hand side of~\eqref{eq: P1_diff_total}, we use~\eqref{eq: P_diff_2_prelim} to write
\begin{align}
    \| \widetilde{P}^*_1 - \widetilde{P}_1 \| &\leq \| R + B^{\top} \widetilde{P}_2 B \| \|\widetilde{K}_1-\widetilde{K}_1^*\|^2 \cr
    &\overset{\mathrm{(i)}}{\leq} \| R + (P_{\max}+a) B^{\top}B \| (\varsigma_{1,\eps})^2 \cr
    &\overset{\mathrm{(ii)}}{\leq} \frac{C_3}{2} \frac{\eps}{4 C_1 C_3} \cr
    &= \frac{\eps}{8 C_1}, \label{eq: P1_diff_2}
\end{align}
where (i) follows from~\eqref{eq: p_diff_leq_a}, and (ii) is due to the definition of $\varsigma_{1,\eps}$ in~\eqref{eq: def_varsigma_h}. Finally, substituting~\eqref{eq: P1_diff_1} and~\eqref{eq: P1_diff_2} in~\eqref{eq: P1_diff_total}, we have
\[
\| P^*_t - \widetilde{P}_t \| \leq \frac{\eps}{8 C_1} + \frac{\eps}{8 C_1} = \frac{\eps}{4 C_1},
\]
thereby establishing~\eqref{eq: P1_req} and concluding the proof of the first claim of Theorem~\ref{thm: LQR_DP_by_delta}. 

Lastly, we need to show if $\eps < \frac{1 - \|A - BK^* \|_*}{\kappa_{P^*}^{1/2} \ \|B\|}$, then $\widetilde{K}_0$ is stabilizing. First, recall that $A-B\widetilde{K}_0$ is stable if $\| A-B\widetilde{K}_0\|_* < 1$. Now note that we can write
\begin{align*}
    \| A-B\widetilde{K}_0\|_* &\le \| A-B K^*\|_* + \| B (K^*-\widetilde{K}_0)\|_* \\
    &= \| A-B K^*\|_* + \left\| P^{*^{1/2}} B (K^*-\widetilde{K}_0) P^{*^{-1/2}}\right\| \\
    &\leq \| A-B K^*\|_* + \|P^{*^{1/2}}\| \|P^{*^{-1/2}}\| \|B\| \|K^* - \widetilde{K}_0 \| \\
    &\leq \| A-B K^*\|_* + \kappa_{P^*}^{1/2} \ \|B\| \eps \\
    &<1,
\end{align*}
which verifies the second claim, concluding the proof of Theorem~\ref{thm: LQR_DP_by_delta}. \proofend

\section{Proof of Theorem~\ref{thm: LQR_DP}: Sketch and Clarification} \label{app: thm_LQR_DP} 
We follow~\cite[App.~C]{XZ-TB:23} up to the recursion that controls the propagated error in the backward pass. First, select the horizon $N$ exactly as in Theorem~\ref{thm: RDE converges to ARE}, which ensures $\|K_0^*-K^*\|\le \eps/2$. Then, repeating their bookkeeping for the outer-loop error terms yields the familiar requirement on the last-stage inner-loop error
\begin{equation}
\label{eq:req-last}
\|e_{N-1}\| \;\lesssim\; \mathcal{O}\left( \sqrt{\frac{\eps}{\,C_1\,C_2^{\,N-2}\,C_3\,}} \right),
\end{equation}
which is precisely the counterpart of eq.~(C.19) in~\cite{XZ-TB:23}.

Substituting the choice of $N$ from~\eqref{eq: N value choice} into \eqref{eq:req-last} gives
\[
C_2^{\,N-2} \approx C_2^{\,\frac{\log(1/\eps)}{2\log(1/\|A^*_K\|_*)}}
= (1/\eps)^{\frac{\log C_2}{2\log(1/\|A^*_K\|_*)}},
\]
and hence
\[
\|e_{N-1}\| \sim \mathcal{O} \left(\eps^{\frac{1}{2}+\frac{\log C_2}{4\log(1/\|A^*_K\|_*)}}\right).
\]
In~\cite{XZ-TB:23}, this exponent is (incorrectly) simplified to $3/4$, i.e., as if $\frac{\log C_2}{4\log(1/\|A^*_K\|_*)}=\frac{1}{4}$ held universally; however, this does not hold, since this ratio is system-dependent in general. The same substitution appears explicitly in\cite[(C.4) and (C.19)]{XZ-TB:23}. Indeed, this error  propagates to their final sample complexity. Since the inner-loop PG (two-point estimation) needs $\widetilde{\mathcal{O}}(\|e_h\|^{-2})$ steps to reach the required accuracy at stage $h$, the overall iteration budget must scale as
\begin{align*}
\widetilde{\mathcal{O}} \left(\max\big\{\|e_0\|^{-2}, \|e_{N-1}\|^{-2}\big\}\right) =\widetilde{\mathcal{O}}\!\left(\max\!\left\{\eps^{-2},\,\eps^{-1-\frac{\log C_2}{2\log(1/\|A^*_K\|_*)}}\right\}\right),
\end{align*}
rather than $\widetilde{\mathcal{O}}(\eps^{-2})$ obtained by substituting $\|e_{N-1}\|=\mathcal{O}(\eps^{3/4})$.

\textbf{Our fix.} Our outer-loop analysis avoids the exponential factor $C_2^{N-2}$ entirely by exploiting that the Riccati operator $\mathcal{R} (\cdot)$ is nonexpansive in the Riemannian distance $\delta (\cdot, \cdot)$ (Lemma~\ref{lem: delta_Riccati_contraction}). This yields stage-wise requirements that do not blow up with $N$, recovering the $\widetilde{\mathcal{O}}(\eps^{-2})$ total rate with a one-point model-free gradient estimator.

% and takes 
% \[
% N \approx \frac{1}{2} \log_{C_2}(1/\eps)
% \]
% to obtain 
% \[
% C_2^{N-2} \approx C_2^N \approx C_2^{\frac{1}{2} \log_{C_2}(1/\eps)} = \frac{1}{\sqrt{\eps}},
% \]
%which is flawed.
% but one can observe that the dependence on $\eps$ is still milder than the requirement for $\|e_{0}\|$. Therefore, it suffices to require error bound for all $t$ to be $\|e_t\| \sim \mathcal{O}(\eps)\mathcal{O}(\texttt{poly}(\text{system parameters)})$ to reach the $\eps$-neighborhood of the infinite-horizon LQR policy. 
% but one can observe that the dependence on $\eps$ is still milder than the requirement for $\|e_{0}\|$. Therefore, it suffices to require error bound for all $t$ to be $\|e_t\| \sim \mathcal{O}(\eps)\mathcal{O}(\texttt{poly}(\text{system parameters)})$ to reach the $\eps$-neighborhood of the infinite-horizon LQR policy. 
% Lastly, for $\widetilde{K}_{0}$ to be stabilizing, it suffices to select a sufficiently small $\eps$ such that the $\eps$-ball centered at the infinite-horizon LQR policy $K^*$ lies entirely in the set of stabilizing policies. A crude bound that satisfies this requirement is
% \begin{align*}
%     \eps < \frac{1 - \|A-BK^*\|_*}{\|B\|} \Longrightarrow \|A-B\widetilde{K}_0\|_* < 1.
% \end{align*}
% This completes the proof.

\end{document}